\documentclass[12pt]{amsart}

\usepackage{latexsym, amssymb}

\newcommand{\be}{\begin{equation}}
\newcommand{\ee}{\end{equation}}
\newcommand{\beq}{\begin{eqnarray}}
\newcommand{\eeq}{\end{eqnarray}}

\newtheorem{prop}{Proposition}[section]
\newtheorem{theo}[prop]{Theorem}
\newtheorem{lemm}[prop]{Lemma}
\newtheorem{coro}[prop]{Corollary}

\newtheorem{defi}[prop]{Definition}

\def\begeq{\begin{equation}}
\def\endeq{\end{equation}}
\def\p{\partial}

\begin{document}

\title {Regularity and rigidity of asymptotically hyperbolic manifolds}

\begin{abstract}
In this paper, we study some intrinsic characterization of
conformally compact manifolds. We show that, if a complete
Riemannian manifold admits an essential set and its curvature tends
to $-1$ at infinity in certain rate, then it is conformally
compactifiable and the compactified metrics can enjoy some
regularity at infinity. As consequences we prove some rigidity
theorems for complete manifolds whose curvature tends to the
hyperbolic one in a rate greater than $2$.
\end{abstract}


\renewcommand{\subjclassname}{\textup{2000} Mathematics Subject Classification}

\author{Xue Hu $^\dag$,  Jie Qing $^\ddag$
and  Yuguang Shi$^\dag$}

\address{Xue Hu, Key Laboratory of Pure and Applied mathematics, School of Mathematics Science, Peking University,
Beijing, 100871, P.R. China.} \email{huxue@math.pku.edu.cn}

\address{Jie Qing, Department of Mathematics, University of California, Santa Cruz, CA 95064, USA} \email{qing@ucsc.edu}

\address{Yuguang Shi, Key Laboratory of Pure and Applied mathematics, School of Mathematics Science, Peking University,
Beijing, 100871, P.R. China.} \email{ygshi@math.pku.edu.cn}

\thanks{$^\dag$ Research partially supported by   NSF grant of China 10725101 and 10990013.}
\thanks{$^\ddag$ Research partially supported by NSF DMS-0700535}

\date{2009}
\maketitle

\markboth{Xue Hu,  Jie Qing  and Yuguang Shi}{}

\section {introduction}
In recent years there are growing interests in the study of
conformally compact Riemannian manifolds from mathematics and
physics. Conformally compact Einstein manifolds, for instance, are
the basic objects that are used to establish the mathematical
foundation for the so-called AdS/CFT correspondence proposed and
studied in some promising theory of quantum gravity.

Suppose that $X^{n+1}$ is a smooth manifold with boundary $\partial
X = M^n$. A defining function $x$ of the boundary $M^n$ in $X^{n+1}$
is a smooth function on $X^{n+1}$ such that
\begin{enumerate}
  \item $x > 0$ in $X^{n+1}$;
  \item $x = 0$ on $M^n$;
  \item $dx \neq 0$ on $M^n$.
\end{enumerate}
A complete Riemannian metric $g$ on $X^{n+1}$ is said to be
conformally compact of regularity $C^{k, \alpha}$ if $x^2g$ extends
to be a $C^{k, \alpha}$ compact Riemannian metric on $\bar X^{n+1}$
for a defining function $x$ of the boundary $M^n$ in $X^{n+1}$.

A basic and interesting question is that, what are the sufficient
conditions for a complete Riemannian manifold $(X^{n+1}, \ g)$ to be
conformally compact of reasonable regularity? It is rather easy to
see that the Riemann curvature needs to tend to a negative constant
 at the infinity. On the other hand, due to the complexity of the
end structure of a hyperbolic manifold, it is clear that a simple
volume growth condition would not be enough to yield anything like
what is true about asymptotically locally Euclidean manifolds.

With those understandings in mind, in this note, as in \cite{ST}, we
will consider a complete, noncompact Riemannian manifold $(X,\ g)$
whose curvature is asymptotically hyperbolic of order $a$ as
follows:
\begin{equation}\label{deacyofcurv}
||Rm-\mathbf{K}||\leq C e^{-a \rho}£¬
\end{equation}
where $Rm$ denotes the Riemann curvature tensor of the metric $g$
and $\mathbf{K}$ the constant curvature tensor of $-1$, i.e.,
$\mathbf{K}_{ijkl}=-(g_{ik}g_{jl}-g_{ij}g_{kl})$, $\rho$ is the
distance function to a fixed point in $X$ with respect to $g$, and
$C$ is a positive constant independent of $\rho$. To control the
wild behavior of the ends we consider a notion of essential sets
which was introduced in \cite{BM} and \cite{Ba}.

\begin{defi}\label{essentialset}
A compact subset $\mathbf{E}$ of $(X, \ g)$ is called an essential
set if
\begin{enumerate}
  \item $\mathbf{E}$ is a compact domain of $X$ with smooth and convex boundary
$\mathbf{B}$ $:= \partial \mathbf{E},$ i.e. its second fundamental
forms with respect to the outward  unit normal vector field is
positive definite. Any geodesic half line emitting from $\mathbf{B}$
orthogonally to the outside of $E$ can be extended to infinity;
  \item the distance function $\rho$ to the essential set is smooth;
  \item the region in $X$ which is outside the essential set
$\mathbf{E}$ is diffeomorphic to $[0,\infty)\times \mathbf{B}$.
\end{enumerate}
\end{defi}

It is not easy to determine whether or not there is an essential set
even in a complete hyperbolic manifold. But it clearly is a
necessary condition for a complete Riemannian manifold to be
conformally compactifiable in the above sense. Indeed in \cite{Ba}
and \cite{BG}, for a complete Riemannian manifold which possesses an
essential set and whose curvature is asymptotically hyperbolic and
covariant derivatives of Riemann curvature decay, the authors were
able to obtain conformal compactifications with some regularity
results. Those decay assumptions on the covariant derivatives of
curvature were derived when in addition the manifold is Einstein
(cf. \cite{BG}).

Now suppose that $(X^{n+1}, \ g)$ is a complete Riemannian manifold
that possesses an essential set $\mathbf{E}$. Then we know that
outside $\mathbf{E}$ the metric $g$ can be written as
$$
g =d\rho^2 + g_{ij}(\rho, \theta)d\theta^i d\theta^j,
$$
where $\rho$ is the distance function to $\mathbf{E}$ and $\theta$
is a local coordinate on $\mathbf{B} = \partial \mathbf{E}$. We have
a convention for indices in this note that all Latin letters runs
from $1$ to $n$ while all Greek letters runs from $0$ to $n$. With
this identification there is a natural differential structure on the
closure $\bar X$ which simply is the product structure as follows.
Let
$$
\mathbf{E}_1 =\{x \in M : \rho(x)\leq 1\},
$$
then $\mathbf{E}\subset \mathbf{E}_1$ and
$$
\bar X =\mathbf{E}_1 \coprod [0,\delta_0]\times \mathbb {B},
$$
where $\delta_0 = \log\frac{e+1}{e-1}$, and $\coprod$ denotes the
connected sum by identifying $(1,\theta)\in$ $\partial\mathbf{E}_1 $
with $(\delta_0, \theta) \in \{\delta_0 \}\times \mathbb {B}$. Hence
we may use the coordinate $(\tau, \theta)$ to replace $(\rho,
\theta)$ outside $\mathbf{E}_1$ such that
$$
\tau=\log\frac{e^\rho +1}{e^\rho -1} \ \text{and} \ \sinh^{-1}\rho =
\sinh \tau.
$$
Therefore, in the new coordinate, we consider
\begin{equation}\label{defbarg}
\begin{split}
\bar{g} & =\sinh^{-2}\rho \cdot g\\
&= \sinh^{-2}\rho d\rho^2+\sinh^{-2}\rho g_{ij}(\rho,
\theta)d\theta^i d\theta^j \\
&=d\tau^2 + \bar g_{ij}(\tau, \theta)d\theta^i d\theta^j,
\end{split}\nonumber
\end{equation}
Now $(X\setminus \mathbf{E}_1, \ \bar g)$ is a conformal
compactification of $(X\setminus \mathbf{E}_1, \ g)$ with the
defining function $\sinh \tau$ of the boundary $\mathbf{B}$ in $\bar
X$. Thus we will focus on the regularity of $\bar g$ at $\tau=0$.
Note that perhaps we will have to consider it in other coordinate
systems nearby the infinity in order to get better regularity.

As pointed out in \cite{BG} the regularity of their conformal
compactifications were obtained via ODE analysis. The draw back for
that is the demand of the assumptions on the decay of covariant
derivatives of Riemann curvature such as:
\begin{equation}\label{deacyofderivativeofcurv}
||\nabla^k Rm||\leq C e^{-b\rho},
\end{equation}
where $C$ is a positive constant independent of $\rho$. For
instance, in \cite {BG}, it assumes (2) for $k=1$ and $b > 0$ to
have $C^{0, b}$ regularity; and it assumes (2) for $k=2$ and $b
> 1$ to have $C^{1, b-1}$ regularity. Note that, if the curvature
condition (\ref{deacyofcurv}) holds with $a > 0$ and the curvature
condition (\ref{deacyofderivativeofcurv}) holds with $b >0$, then
$a\geq b$, as noticed in \cite{BG}. Our first goal in this note is
to make use of harmonic coordinates near the infinity. We are able
to obtain the $C^{2, \alpha}$ regularity under the curvature
condition (\ref{deacyofderivativeofcurv}) with $k=1$ and $b > 2$.

\begin{theo}\label{mainresult3}
Suppose that $(X^{n+1}, \ g)$ is a complete Riemannain manifold with
an essential set $\mathbf{E}$ and that it satisfies the curvature
condition (\ref{deacyofcurv}) with $a > 0$ and the curvature
condition (\ref{deacyofderivativeofcurv}) with $k=1$ and $b > 2$.
Then there is differentiable structure on $\bar X$, which is smooth
in the interior of $X$ and $C^{3, \alpha}$ up to the boundary for
some $\alpha \in (0,1)$. And in this differentiable structure, $\bar
g$ is smooth in the interior of $X$ and is $C^{2, \alpha}$ smooth up
to the boundary for some $\alpha \in (0,1)$.
\end{theo}

We remark here that in fact we can get $C^{1, \alpha}$ regularity
when $b>2 - \frac 1{n+1}$ (please see Theorem \ref{mainresult1} in
Section 3). Our second goal in this note is to make use of the Ricci
flow on complete manifolds to obtain some regularity without
assuming decay conditions on the covariant derivatives of curvature.
Let $(X^{n+1}, \ g)$ be a complete Riemannian manifold satisfying
the curvature condition (\ref{deacyofcurv}) with $a > 0$. We
consider the normalized Ricci flow as follows:
\begin{equation}\label{nRicciflow}
\left\{
\begin{array}{ll}
    \frac{\partial}{\partial t}g_{\alpha\beta} = -2ng_{\alpha\beta}-2R_{\alpha\beta},  \\
    g_{\alpha\beta}(x,0)=g_{\alpha\beta}(x), \quad \text{on}\quad X.\\
\end{array}
\right.
\end{equation}
Since our initial metric satisfies the curvature condition
(\ref{deacyofcurv}), by the works of Shi (see Theorem 1.1 in
\cite{Shi}), the evolution equation $(\ref{nRicciflow})$ has a
smooth solution $g(\cdot, t)$ for a short time. Let
$$
E_{\alpha\beta\gamma\delta}(g) = R_{\alpha\beta\gamma\delta}(g) +
(g_{\alpha\gamma}g_{\beta\delta}-g_{\alpha\delta}g_{\beta\gamma}).
$$
Using the evolution equations and the maximum principles we show
that

\begin{lemm}
Suppose that $(X^{n+1}, \ g)$ is a complete Riemannian manifold and
that it satisfies the curvature condition (\ref{deacyofcurv}) with
$a > 0$. Let $(M,g(t))$, $t\in [0,T],$ be a complete solution of the
normalized Ricci flow (\ref{normalizedequation}). Then there exists
constants $T_{0}$ and $C$ such that
$$
\|E(g(t))\|\leq Ce^{-a\rho_{0}}
$$
and
$$
\|\nabla E(g(t))\|\leq \frac{C}{\sqrt{t}}e^{-a\rho_{0}},$$ where
$\nabla$ is with respect to metric $g(t)$, and $C$ is independent of
$t$, $0<t\leq$ $T_{0}\leq T.$
\end{lemm}

It is then rather easy to obtain an improvement of Theorem A in
\cite{BG} as follows:

\begin{theo}\label{holdercontinuous} Suppose $(X^{n+1}, \ g)$ is a complete
Riemannian manifold with an essential set $\mathbf{E}$ and that it
satisfies the curvature condition (\ref{deacyofcurv}) with $0<a<1$.
Then $\bar {g}$ is $C^{0,\mu}$ smooth up to the boundary $\tau =0$
for $\mu = \frac 23 a$.
\end{theo}

We observe that the curvature estimates for the compactified metrics
$\bar g$ depends only on the curvature condition
(\ref{deacyofcurv}), even though our constructions of harmonic
coordinates used the curvature condition
(\ref{deacyofderivativeofcurv}). Based on compactness theorems in
\cite {GP} of Riemannian manifolds we find the regularity of
harmonic coordinates indirectly via some good approximations
provided by the Ricci flow. We then obtain
\newpage
\begin{theo}\label{mainresult2}
Suppose $(X^{n+1}, \ g)$ is a complete Riemannian manifold with an
essential set $\mathbf{E}$ and that it satisfies the curvature
condition (\ref{deacyofcurv}) with $a > 2-\frac 1{n+1}$. Then there
is differentiable structure $\Gamma$ on $\bar X$, which is smooth in
the interior of $X$ and $W^{3,p}$ up to the boundary for some $p
> n+1$. And in this differentiable structure $\Gamma$, $\bar g$ is
smooth in the interior of $X$ and it is $W^{2,p}$ up to the boundary
for some $p>n+1$, hence, it is $C^{1, \mu}$ smooth up to the
boundary for some $\mu \in (0,1)$. Moreover $\bar g$ is smooth in
the interior of $X$ and it is $W^{2,p}$ up to the boundary for any
$p>1$, hence, it is $C^{1, \mu}$ smooth up to the boundary for any
$\mu \in (0,1)$ when $a \geq 2$.
\end{theo}

One of the motivation to derive an intrinsic criterion for a
complete Riemannian manifold to be conformally compactifiable is to
find a rigidity theorem for a complete Riemannian manifold whose
curvature is asymptotically hyperbolic as in the work of Shi and
Tian \cite{ST}. Although many interesting rigidity theorems were
obtained (e.g. \cite{AD}, \cite{BMQ},\cite{Bi},\cite{DJ} \cite{Q1}
\cite{ST}) lately, most of them need to assume some regularity of
the comformally compactness of manifolds at infinity.  As
consequences of our regularity theorems we will state and prove
Theorem \ref{Einstein} and Theorem \ref{Riccicomparison} in \S 5.
Theorem \ref{Einstein} is a rigidity theorem for Einstein AH
manifolds; while Theorem \ref{Riccicomparison} requires the
curvature condition (2). And in both cases the rigidity in higher
dimensions requires the spin condition.

Very recently, in \cite{DJ}, a rigidity theorem for asymptotically
hyperbolic manifold with $Ric \geq - n g$ and admitting $C^2$
conformal compactification was proved. We find that our curvature
estimate (\ref{deacyofcurv}) in Proposition
\ref{estimatesofcompactifiedcurvature} and the above Theorem
\ref{mainresult2} together is a perfect substitute for the
regularity assumption in the rigidity theorem of \cite{DJ}. Our main
rigidity theorem is as follows:

\begin{theo}\label{rigidity}
Let $(X^{n+1}, \ g)$ be a complete manifold with $Ric \geq -ng$.
Suppose that it has an essential set $\mathbf E$ and it satisfies
the curvature condition (\ref{deacyofcurv}) with $a > 2$. Assume
also that $X^{n+1}$ is simply connected at the infinity. Then
$(X^{n+1}, \ g)$ is a standard hyperbolic space for $n\geq 4$.  And
it is a standard hyperbolic space if in addition we assume that
$\int_X \|Rm -\mathbf{K}\|d\mu_g <\infty$ for $n=3$.
\end{theo}

The rest of the paper is organized as follows. In \S2, we will
present some basic estimates which can be derived mostly just from
the Riccati equations and ODE analysis. In \S3, we construct
harmonic coordinates at the infinity, therefore prove Theorem
\ref{mainresult3} via the estimates for the curvature of the
compactified metrics. In \S4, we drop the curvature condition
(\ref{deacyofderivativeofcurv}) by using Ricci flow. Finally in \S5,
we prove rigidity results.

{\bf Acknowledgements} The authors would like to thank Professor
Gang Tian for his encouragement and many enlightening discussions.
The second and third named authors would like to thank
Mittag-Leffler institute for its hospitality and the wonderful
environment where part of the research in this paper was conducted
when they attended the program in mathematical relativity.

\section {Basic estimates}

In this section we present the basic estimates. First we introduce
some estimates that are consequences of the Riccati equations via
ODE analysis as Lemma 2.3 in \cite{ST} (see also in \cite{Ba},
\cite{BG}). Suppose that $(X^{n+1}, \ g)$ is a complete Riemannian
manifold with an essential set $\mathbf{E}$. And suppose that
(\ref{deacyofcurv}) holds. Let $\Sigma_\rho$ be the level surface of
the distance function $\rho$ to the essential set $\mathbf{E}$. For
simplicity we may take orthonormal frames on each slice
$\Sigma_\rho$, under which the second fundamental forms are denoted
by $h_{ij}$, then we recall Riccati's equations:
$$
\frac{\p h_{ij}}{\p \rho} +h_{ik}h_{kj}=R_{0i0j},
$$
where the index $0$ refers to unit normal direction of
$\Sigma_\rho$.

\begin{lemm}\label{estforriccatiequa}
Suppose that $f(\rho) \geq \frac 14$ is a smooth function for $\rho
> 0$ and that
$$
|f(\rho)-1|\leq L e^{-a \rho},
$$
for any $\rho>0$ and some $a, L >0$. If $y$ is the solution of the
equation
$$
\left\{
\begin{array}{ll}
    y' +y^2 =f  \\
    y(0)>0 \\
\end{array},
\right.
$$
\\
then there is a positive constant $C$ which depends only on $L$ and
$y(0)$ such that

\begin{enumerate}
    \item $|y-1|\leq C e^{-2\rho}$, if $a >2$;
    \item  $|y-1|\leq C \rho e^{-2\rho}$, if $a=2$;
    \item $|y-1|\leq C e^{-a \rho}$, if $0< a <2$.
\end{enumerate}
\end{lemm}

\begin{proof} The proof is very similar to the proof of
Lemma 2.3 in \cite{ST}. First it is easy to see that
\begin{equation}\label{1bd}
0<y<C,
\end{equation}
for some constant $C$, which only depends on $L$ and $y(0)$. Let
$$
y=v+1.
$$
Then we get
$$
v' +2v =(f-1)-v^2
$$
and
$$
(v^2)'+2(2+v)\cdot v^2 =2v\cdot (f-1).
$$
Hence
$$
(v^2)' +2 v^2 <(v^2)'+2(2+v)\cdot v^2 =2v\cdot (f-1).
$$
Therefore, in the light of (\ref{1bd}) and the assumptions on the
function $f$, it follows that
$$
|v|^2\leq C e^{-2\rho}, ~~~if~a>2,
$$
$$
|v|^2\leq C \rho e^{-2\rho}, ~~~if~a=2,
$$
$$
|v|^2\leq C e^{-a\rho}, ~~~if~0<a<2.
$$
Plugging back those into the same equation we have the improved
estimates for $v^2$. Finally we finish the proof of Lemma
\ref{estforriccatiequa} using the equation
$$
v' +2v =(f-1)-v^2.
$$
\end{proof}

Consequently, if we write
$$
h_{ij}=\delta_{ij} + T_{ij}e^{-2\rho},
$$
we have

\begin{equation}\label{2ndfundamentalest1}
|T_{ij}|\leq C, \quad \text{if} \quad a > 2;
\end{equation}
\begin{equation}\label{2ndfundamentalest2}
|T_{ij}|\leq C\rho, \quad \text{if} \quad a = 2;
\end{equation}
and
\begin{equation}\label{2ndfundamentalest3}
|T_{ij}|\leq C e^{(2-\alpha)\rho}, \quad \text{if} \quad 0<  a < 2,
\end{equation}

Note that the constants $C$ in the above depend only on the second
fundamental forms of the level surface $\Sigma_0$ and the constant
in the assumption $(\ref{deacyofcurv})$. Similarly we also have the
estimates of the second fundamental forms $S_{ij}$ of $\Sigma_\rho$
under local coordinates $\{\frac{\p}{\p \theta^i}\}_{i=1}^{n}$ on
$\mathbf{B}$. Again if we write
$$
S^j _i =g^{kj} S_{ik},
$$
and
$$
S^j _i =\delta^j _i + p^j _ie^{-2\rho},
$$
then
\begin{equation}\label{estimatesp1}
|p^j _i|\leq C, \quad \text{if} \quad a>2;
\end{equation}

\begin{equation}\label{estimatesp2}
|p^j _i|\leq C\rho, \quad \text{if} \quad a=2;
\end{equation}
and
\begin{equation}\label{estimatesp3}
|p^j _i|\leq C e^{(2-\alpha)\rho}, \quad \text{if} \quad 0< a <2.
\end{equation}

Using those estimates on the second fundamental form of the level
set the distance function $\rho$ one can easily derive the following
estimates on the curvature of the compactified metric $\bar g$ as
follows:

\begin{prop}\label{estimatesofcompactifiedcurvature}
Let $(X^{n+1}, \ g)$ be an complete Riemannian manifold with an
essentail set $\mathbf{E}$ and satisfying the curvature condition
(\ref{deacyofcurv}) for $a > 0$. If
$g=d\rho^{2}+g_{ij}d\theta^{i}d\theta^{j}$, and
$\bar{g}=\sinh^{-2}\rho
g:=d\tau^{2}+\bar{g}_{ij}d\theta^{i}d\theta^{j}$, then we have
\begin{enumerate}
  \item $\| Rm(\bar g) \|_{\bar g}\leq \Lambda$, if $a >2;$
  \item $\| Rm(\bar g) \|_{\bar g}\leq \Lambda \rho $, if $a =2;$
  \item $\| Rm(\bar g) \|_{\bar g}\leq \Lambda e^{(2-a)\rho}$, if $0< a
  <2.$
\end{enumerate}
If in addition we assume the curvature condition
(\ref{deacyofderivativeofcurv}) for $k=1$ and $1 < b < 3$, then we
also have
$$
\|\bar \nabla  Rm(\bar g)\|_{\bar g} \leq \Lambda e^{(3-b)\rho}.
$$
Here $\Lambda$ is a constant depending only on $C$ in the
assumptions (\ref{deacyofcurv}) and (\ref{deacyofderivativeofcurv}).
\end{prop}

\begin{proof} Let $\bar{S}$ be the
second fundamental form of $\Sigma_\rho$ in $(X^{n+1}, \ \bar g)$.
Since there exists a constant $\Lambda$ such that in local
coordinates $\{\mathcal{U},(\rho,\theta^{1},...,\theta^{n})\},$
$$
\frac{1}{\Lambda}e^{2\rho}\delta_{ij}\leq g_{ij}\leq \Lambda
e^{2\rho}\delta_{ij},
$$
we have
$$
|\bar{g}_{ij}|\leq C.
$$
Then by a direct computation we get
$$
\p_{\tau}\bar{g}_{ij}=2(\frac{\cosh\rho}{\sinh^{2}\rho}g_{ij}-\frac{S_{ij}}{\sinh\rho})=-
2\bar{S}_{ij}.
$$
Hence we have
$$
\bar{S}_{ij}=\frac{e^{-2\rho}}{\sinh\rho}p_{ij} - e^{-\rho}\bar
g_{ij}
$$
and
$$
\bar S_i^j = e^{-2\rho}\sinh\rho p_i^j - e^{-\rho}\delta_i^j.
$$
Using Gauss identity, we want to express $\bar{R}_{ijk}^{l}$ in
terms of $E_{ijk}^{l}$ and $p_{i}^{j}$ as follows:
$$
\bar{R}_{ijk}^{l}= e^{-2\rho}(\delta_{j}^{l}\bar
g_{ik}-\delta_{i}^{l}\bar g_{jk})+ E_{ijk}^{l} - e^{-2\rho}\coth\rho
(p_{ik}\delta_{j}^{l}+p_{j}^{l}g_{ik}-p_{i}^{l}g_{jk}-p_{jk}\delta_{i}^{l}).
$$
On the other hand, we may calculate directly
$$
\bar{R}_{0jk}^{l} = - \sinh\rho \ E_{0jk}^{l}
$$
and
$$
\bar{R}_{0i0}^{k} = \sinh^{2}\rho\cdot E_{0i0}^{k}- e^{-2\rho}\cdot
\sinh\rho\cdot \cosh\rho\cdot p_{i}^{k}+e^{-\rho}\cdot \cosh\rho
\cdot \delta_{i}^{k}.
$$
Those readily imply the estimates for the curvature of the
compactified metric $\bar g$ with the estimates (\ref{estimatesp1}),
(\ref{estimatesp2}) and (\ref{estimatesp3}).

To obtain the estimate on the normal derivatives of the second
fundamental form $\bar S$, we derive from the Riccati equations that
$$
\p_\rho p_{i}^{j}=e^{2\rho}E_{0i0}^{j}-e^{-2\rho}p_{i}^{k}p_{k}^{j},
$$
which implies
\begin{equation}\label{estimateNderivatives}
|\p_\rho p_{i}^{j}|\leq \Lambda e^{(2-a)\rho}.
\end{equation}

\noindent To obtain the estimate on the tangential derivatives of
the 2nd fundamental form $\bar S$, we will take derivatives in the
two sides of the Riccati's equation. But first
$$
R^j_{0i0,k} = \p_{k}R^j_{0i0} - R^j_{li0}\Gamma^l_{0k} -
R^j_{0l0}\Gamma^l_{ik} -R^j_{0il}\Gamma^l_{0k} +
R^l_{0i0}\Gamma^j_{lk},
$$
which implies
$$
|\p_{k}R^j_{0i0}| \leq \Lambda e^{-(b-1)\rho},
$$
where we have used the fact that the compactified metric $\bar g$ is
Lipschitz when $k=1$ and $b>1$ in (\ref{deacyofderivativeofcurv})
and $a>0$ in (\ref{deacyofcurv}) (cf. \cite{Ba} and \cite {BG}).
Then from Riccati's equations, we have
$$
\p_{\rho}\p_{k} S^j_i+\p_{k} S^l_i\cdot S^j_l +\p_{k} S^j_l\cdot
S^l_i=\p_{k}R^j_{0i0}.
$$
Considering the ODE system
$$
y'+(2+\Omega)y = O(e^{-(b-1)\rho}),
$$
where
\begin{eqnarray}
|\Omega|\leq \left\{\begin{array}{l}\Lambda e^{-2\rho}, a > 2,\\
\Lambda \rho e^{-2\rho},~~ a = 2,\\
\Lambda e^{-a\rho},~~ 0 < a <2,
\end{array}
\right.\nonumber
\end{eqnarray}\\
and an argument similar to the one in the proof of Lemma
\ref{estforriccatiequa}, we then have
\begin{equation}\label{tangentialderivestimate2ndfundfoms}
|\p_{k} S^j_i|\leq \Lambda e^{-(b-1)\rho},
\end{equation}
\noindent where $\Lambda$ is a positive constant independent of
$\rho$ and $\theta$. Hence

\begin{equation}\label{estimateTderivatives}
|\p_k p_{i}^{j}|\leq \Lambda e^{-(b-1)\rho}.
\end{equation}

Now we are ready to get the estimates on the covariant derivatives
of curvature by direct computations.

\begin{equation}
\begin{split}
\bar{R}_{ijk,m}^{l}&=E_{ijk,m}^{l}-\frac{\cosh\rho}{\sinh\rho}
(E_{0kji}\delta_{m}^{l}-E_{0jk}^{l}g_{mi}+E_{0ik}^{l}g_{mj}-E_{ij0}^{l}g_{mk})\\
&-\frac{e^{-2\rho}\cosh\rho}{\sinh\rho}(\p_{m}p_{ik}
\delta_{j}^{l}-\p_{m}p_{jk}\delta_{i}^{l}+p_{j}^{n}g_{ik}\Gamma_{mn}^{l}-p_{i}^{n}g_{jk}
\Gamma_{mn}^{l}\\
&-p_{nk}\delta_{j}^{l}\Gamma_{mi}^{n}+p_{nm}\delta_{i}^{l}\Gamma_{mj}^{n}+
p_{n}^{l}g_{jk}\Gamma_{mi}^{n}
-p_{n}^{l}g_{ik}\Gamma_{mj}^{n}\\
&-p_{in}\delta_{j}^{l}\Gamma_{mk}^{n}+p_{jn}\delta_{i}^{l}\Gamma_{mk}^{n})
\end{split}
\end{equation}

\begin{equation}
\begin{split}
\bar{R}_{ijk,0}^{l} & = - \sinh\rho E_{ijk,0}^{l}-2\cosh\rho
E_{ijk}^{l}+\frac{2e^{-3\rho}}{\sinh\rho}(\delta_{j}^{l}g_{ik}-\delta_{i}^{l}g_{jk})\\
& + (\frac{e^{-4\rho}}{\sinh\rho} -
2e^{-3\rho}\cosh\rho)(p_{ik}\delta_{j}^{l} + p_{j}^{l} g_{ik}-
p_{i}^{l} g_{jk}- p_{jk} \delta_{i}^{l}) \\
& + 2e^{-2\rho}\cosh\rho (p_{j}^{l}g_{ik}- p_{i}^{l} g_{jk}) -
2e^{-4\rho}\cosh\rho (p_{mk}p_{i}^{m}\delta_{j}^{l}
- p_{mk} p_{j}^{m} \delta_{i}^{l})\\
& + e^{-2\rho} \cosh\rho (\p_{\rho} p_{ik} \delta_{j}^{l} +
\p_{\rho} p_{j}^{l} g_{ik} - \p_{\rho} p_{i}^{l} g_{jk} - \p_{\rho}
p_{jk} \delta_{i}^{l})
\end{split}
\end{equation}

\begin{equation}
\begin{split}
\bar{R}_{0jk,m}^{l} & = -\sinh\rho\cdot
E_{0jk,m}^{l} + \cosh\rho\cdot(E_{0jk}^{0}\delta_{m}^{l}-E_{mjk}^{l} - E_{0j0}^{l} g_{mk})\\
& + (\frac{e^{-3\rho}}{\sinh^{2}\rho} -
\frac{e^{-2\rho}\cosh\rho}{\sinh^{2}\rho}) g_{jk}\delta_{m}^{l}
- (\frac{e^{-3\rho}}{\sinh^{2}\rho} + \frac{e^{-2\rho}\cosh\rho}{\sinh^{2}\rho}) g_{mk}\delta_{j}^{l}\\
& + (\frac{e^{-4\rho}}{\sinh\rho} + \frac{2e^{-3\rho}\cosh\rho}{\sinh\rho}) p_{mk}\delta_{j}^{l}\\
& - \frac{e^{-4\rho}}{\sinh\rho} g_{jk}p_{m}^{l} + \frac{2e^{-3\rho}\cosh\rho}{\sinh\rho} g_{mk}p_{j}^{l}\\
& - e^{-4\rho}\cosh\rho (p_{nk}p_{m}^{n} \delta_{j}^{l} -
g_{jk}p_{m}^{n} p_{n}^{l} + 2p_{jb}^{l} p_{mk}).
\end{split}
\end{equation}

\begin{equation}
\bar{R}_{0jk,0}^{l} = \sinh^{2}\rho\cdot E_{0jk,0}^{l}+ 2
\sinh\rho\cdot\cosh\rho\cdot E_{0jk}^{l}
\end{equation}

\begin{equation}
\begin{split}
\bar{R}_{0j0,m}^{l} & = \sinh^{2}\rho\cdot E_{0j0,m}^{l} +
\sinh\rho\cdot \cosh\rho (E_{mj0}^{l}
+ E_{0jm}^{l}) \\ & - e^{-2\rho}\sinh\rho\cdot \cosh\rho\cdot\p_{m}p_{j}^{l}\\
& + e^{-2\rho}\sinh\rho\cdot \cosh\rho(p_{m}^{l}\Gamma_{nj}^{m} -
p_{j}^{m} \Gamma_{nm}^{l})
\end{split}
\end{equation}

\begin{equation}
\begin{split}
\bar{R}_{0j0,0}^{l} & = - \sinh^{3}\rho\cdot E_{0j0,0}^{l} -
2\sinh^{2}\rho\cdot \cosh\rho E_{0j0}^{l}
\\ & +(e^{-2\rho}\sinh\rho\cdot\cosh^{2}\rho + e^{-2\rho} \sinh^{3}\rho
- 2 e^{-3\rho}\sinh^{2}\rho\cdot\cosh\rho) p_{j}^{l}\\
& + e^{-2\rho}\sinh^{2}\rho\cdot \cosh\rho\cdot \p_{\rho}p_{j}^{l} +
e^{-2\rho}\sinh\rho\delta_{j}^{l}
\end{split}
\end{equation}

\noindent Combine above calculations with
(\ref{estimateNderivatives}), (\ref{estimateTderivatives}) and
assumption (\ref{deacyofderivativeofcurv}), we arrive at
$$
\| \bar \nabla  Rm(\bar g)\|_{\bar g} \leq \Lambda e^{(3-b)\rho}.
$$
Thus we complete the proof of the proposition.
\end{proof}

Consequently we have

\begin{prop}\label{basicestimate}
Suppose $(X^{n+1}, \ g)$ is a complete Riemannain manifold with an
essential set $\mathbf{E}$, and that it satisfies the curvature
condition (\ref{deacyofcurv}) with $a > 0$. Then
\begin{enumerate}

    \item \label{1stnormderiestimate} Near the boundary $\tau=0$, if $ a > 2$, then
    $|\frac{\p \bar g_{ij}}{\p \tau}|\leq C \tau$; if $a = 2$, then $|\frac{\p \bar g_{ij}}{\p \tau}|\leq C
    |\tau\log\tau|$; if $0< a <2$, then $|\frac{\p \bar g_{ij}}{\p \tau}|\leq C
    \tau^{\alpha-1}$.

    \item   \label{2ndnormderiestimate} Near the boundary $\tau=0$, if $a > 2$,
    then $|\frac{\p^2 \bar g_{ij}}{\p \tau^2}|\leq C $; if $a = 2$,
    then $|\frac{\p^2 \bar g_{ij}}{\p \tau^2}|\leq C |\log\tau|$; if $0< a <2$, then $|\frac{\p^2 \bar g_{ij}}{\p
    \tau^2}|\leq C \tau^{a - 2}$.

    \item \label{Lip} $\bar g_{ij}(\tau, \theta)$ is lipschitz up to the boundary $\tau=0$,
    if the condition $(\ref{deacyofderivativeofcurv})$ holds with $k=1$ and $b >1$.

    \item \label{2mixderiestimate} Near the boundary $\tau=0$, $|\frac{\p^2 \bar g_{ij}}{\p \tau \p
    \theta^l}|\leq C \tau^{b-2} $, if the curvature condition $(\ref{deacyofderivativeofcurv})$ with $k=1$ and $b > 1$.

\end{enumerate}
\end{prop}

\begin{proof}
For (\ref{1stnormderiestimate}), recall that
\begin{equation}
\p_\tau \bar g_{ij} = 2e^{-\rho}\bar g_{ij} - 2\frac
{e^{-2\rho}}{\sinh\rho} p_{ij}. \nonumber
\end{equation}
Hence it is easily seen that (\ref{1stnormderiestimate}) holds in
the light of (\ref{estimatesp1}), (\ref{estimatesp2}),
(\ref{estimatesp3}) and the fact that
$$
\tau = \log (1 + \frac 2{e^\rho -1}) \sim e^{-\rho}, \quad\text{as
$\rho\to\infty$}.
$$
As for (\ref{2ndnormderiestimate}) we calculate
$$
\frac{\partial^2\bar{g}_{ij}}{\partial\tau^{2}}  =
2(e^{-\rho}\p_\tau \bar g_{ij} + e^{-\rho}\sinh\rho\bar g_{ij}) -
2(2e^{-2\rho}p_{ij}  + e^{-2\rho}coth\rho p_{ij} - e^{-2\rho}\p_\rho
p_{ij})
$$
Hence (\ref{2ndnormderiestimate}) in this proposition holds in the
light of (\ref{estimateNderivatives}). (\ref{Lip}) in this
proposition was proved in \cite{Ba}. Therefore we have only
(\ref{2mixderiestimate}) in this proposition left to be proven.
Again we calculate
$$
\p_\tau\p_{\theta^l}\bar g_{ij} = 2 e^{-\rho}\p_{\theta^l}\bar
g_{ij} - 2\frac{e^{-2\rho}}{\sinh\rho}\p_{\theta^l}p_{ij}.
$$
Thus (\ref{2mixderiestimate}) in this proposition is proven due to
(\ref{estimateTderivatives}).
\end{proof}

So far we only employed ODE analysis and Riccati's equations to
derive estimates on the compactified metric $\bar g$. To end this
section we include a simple fact of calculus for later use.

\begin{lemm}\label{elementarylemma}
Suppose $f(x,y)$ is a function on $\mathbf{R}^{n-1} \times [0,
+\infty)$ and that
$$
|\nabla f|\leq C y^{-\delta},
$$
for some $\delta\in (0, 1)$ and a positive constant $C$. Then there
is a constant $\Lambda$ that depends only on $C$ such that
$$
\| f\|_{C^{0, 1-\delta}(\mathbf{R}^{n-1} \times [0, +\infty))} \leq
\Lambda.
$$
\end{lemm}

\begin{proof}
For any two points $(x_1, y_1), (x_2, y_2) \in \mathbf{R}^{n-1}
\times [0, +\infty)$ with $y_2 \geq y_1$, we consider the following
two cases:

\noindent{\bf Case 1}: Suppose $y_2 \geq |x_1 -x_2|$, then we have

\begin{equation}
\begin{split}
|f(x_1, y_1)-f(x_2, y_2)|&\leq |f(x_1, y_1)-f(x_1, y_2)|+|f(x_1,
y_2)-f(x_2, y_2)|\\
&\leq \frac{C}{1-\delta}(y^{1-\delta}_2 -y^{1-\delta}_1) +C
y^{-\delta}_2
|x_1 - x_2|\\
&\leq \frac{C}{1-\delta}|y_1 -y_2|^{1-\delta} +C|x_1 -
x_2|^{1-\delta},
\end{split}
\end{equation}
\noindent hence in this case the lemma is proven.

\noindent{\bf Case 2}: Suppose $y_2 \leq |x_1 -x_2|$, in this case,
take $y_0= |x_2 -x_1|$, then we have

\begin{equation}
\begin{split}
|f(x_1, y_1) & -f(x_2, y_2)|\leq |f(x_1, y_1)-f(x_1, y_0)|+|f(x_1,
y_0)-f(x_2, y_0)|\\
&+ |f(x_2, y_0)-f(x_2, y_2)|\\
& \leq \frac{C}{1-\delta}(y^{1-\delta}_0 -y^{1-\delta}_1) +
\frac{C}{1-\delta}(y^{1-\delta}_0 -y^{1-\delta}_2)+ C y^{-\delta}_0
|x_1 - x_2|\\
&\leq \frac{C}{1-\delta}y^{1-\delta}_0 +C|x_1 - x_2|^{1-\delta}\\
&\leq \Lambda |x_1 - x_2|^{1-\delta}
\end{split}
\end{equation}
\noindent Thus, we see in both cases Lemma is true.
\end{proof}

\section {Harmonic coordinates at infinity}

From the previous section we know that the normal derivatives of the
compactified metric $\bar g$ is well under control when the
curvature condition (\ref{deacyofcurv}) holds with reasonably large
$a > 0$, and so is the curvature of $\bar g$. But to use only ODE
analysis to control the metric $\bar g$, even just in $C^{0,
\alpha}$ norm, one needs to assume curvature condition
(\ref{deacyofderivativeofcurv}) (cf. \cite{Ba} and \cite{BG}). In
this section we will make use of elliptic PDE to improve the basic
estimates established in the previous section. Our approach is very
straightforward. We want to construct a harmonic coordinate system
near by the infinity to translate the estimates of curvature to the
better regularity of the metric $\bar g$. Particularly when $a
> 1$ the boundary $\mathbf{B}^n =
\partial X^{n+1}$ is totally geodesic with respect to the compactified
metric $\bar g$. Recall
$$
\bar g =d\tau^2 +\bar g_{ij}(\tau, \theta)d\theta^i d\theta^j
$$
in a local coordinate near by the infinity, $(0, \epsilon]\times
\mathbf{B}$ for some small $\epsilon$, as we chose before. Therefore
we can build a double $N = [-\epsilon, \epsilon]\times \mathbf{B}$
of the manifold $(0, \epsilon]\times \mathbf{B}$ and extend the
metric $\bar g$ to the double $N$ evenly. It is clear that the
doubled metric $\bar g$ is lipschitz when the curvature condition
(\ref{deacyofcurv}) holds with $a >0$ and the curvature condition
(\ref{deacyofderivativeofcurv}) holds with $k=1$ and $b > 1$. We now
start to construct a harmonic coordinate in the double $N$ across
the boundary $\{0\}\times\mathbf{B}$.

\begin{lemm}\label{approxharmoniccoord}
Suppose that $(X^{n+1}, \ g)$ is a complete Riemanaian manifold with
an essential set $\mathbf{E}$. Suppose that the curvature condition
(\ref{deacyofcurv}) holds with $a >0$ and the curvature condition
(\ref{deacyofderivativeofcurv}) holds with $k=1$ and $b > 1$. Let
$(\tau=\theta^0, \theta^1, \theta^2, \cdots \theta^n)$ be a local
coordinate system out of the product structure near the infinity of
$X$ as before. Then there exists a constant $C$ independent of
$\tau$ and $\theta$ such that on a subset $[-\tau_0, \tau_0]\times
\mathcal{O}$ of the double $N$, we have
$$
|\bar \Delta  \theta^\alpha|\leq C, \quad \forall \ \alpha = 0, 1,
\cdots, n,
$$
where $\mathcal{O}$ is an open set in $\mathbf{B}$, $\tau_0$ is some
small number and $\bar \Delta $ is the Laplacian operator with
respect to the metric $\bar g$.
\end{lemm}

\begin{proof} This is a simply consequence of the fact that the
metric $\bar g$ is Lipschitz when the curvature condition
(\ref{deacyofcurv}) holds with $a >0$ and the curvature condition
(\ref{deacyofderivativeofcurv}) holds with $k=1$ and $b > 1$ by the
work in \cite{Ba} and \cite{BG}.
\end{proof}

Let $\phi=(\theta^0, \theta^1, \cdots, \theta^n)$ and $D_{3\tau_1}
\subset (-\tau_0, \tau_0)\times \mathcal{O}$ be a geodesic ball in
the double $N$ centered on the boundary $\{0\}\times\mathbf{B}$, we
consider the following Dirichlet problem:
\begin{equation}\label{Diriproblem}
\left\{
\begin{array}{ll}
    \bar \Delta \psi =0, \quad \text{in}  \quad D_{3\tau_1}, \\
    \psi|_{\p D_{3\tau_1}} =\phi|_{\p D_{3\tau_1}}, \\
\end{array}
\right.
\end{equation}
Let $\xi = \psi -\phi$ and $\theta=\tau_{1}z$. Then we have
\begin{equation}\label{Diriproblem}
\left\{
\begin{array}{ll}
    \mid \bar\Delta_{z}\xi \mid\leq C\tau_{1}^2 \quad \text{in}  \quad D_{3} \\
    \xi|_{\p D_{3}} =0. \\
\end{array}
\right.
\end{equation}
Hence by $W^{2,p}$ interior estimates we obtain
$$
\|\xi\|_{W^{2,p}(D_{2})} \leq C \tau^{2}_1,
$$
where $p>1$ is arbitrary and $C$ depends on $n$, $p$ and the
constant of lemma \ref{approxharmoniccoord}. Due to the Sobolev
embedding theorem we have
$$
\|\xi\|_{C^{1,\mu}(D_{2})} \leq C \tau^{2}_1, ~~~if p > n+1,
$$
here $\mu=1-\frac{n}{p}$. By rescaling back to $x$-variable,
$$
\|\xi\|_{C^{1,\mu}(D_{2\tau_1})} \leq C \tau^{1-\mu}_1, ~~~if p > n
+ 1,
$$
here $C$ is a constant independent of $\tau$, $\tau_1$ and $\theta$.
Therefore, by choosing $\tau_1$ sufficiently small we see $\psi$ is
harmonic coordinates in $D_{2\tau_1}$ and there exists a constant
$\delta_{0}>0$ independent of $\tau,\tau_{1},$ and $\theta$ such
that $|\det (D\psi)|\geq \delta_{0}$ for all point in $D_{2\tau_1}$.

Let $(D_{2\tau_{1}},y^\gamma)$, $0\leq \gamma \leq n$, be the
harmonic coordinates we just constructed, in the following, we try
to get some higher order smoothness for those coordinates near the
boundary. First we take the advantage of (\ref{2mixderiestimate}) in
Proposition \ref{basicestimate}. Namely,

\begin{lemm}\label{higherordersmooth}
Suppose that $(y^0, y^1, \cdots, y^n)$ is the harmonic coordinates
constructed above. Then
$$
\| \frac{\p}{\p \tau} y^\gamma\|_{C^{1,\mu}(D_{\frac{3}{2}\tau_1})}
\leq C,
$$
where $\mu=1-\frac{n+1}{p}$, $C$ is a constant, and $p$ satisfies
\begin{enumerate}
    \item \label{greaterthantwo} $p > n+1$
    if the curvature condition (\ref{deacyofcurv}) holds with $a >0$ and the curvature
    condition (\ref{deacyofderivativeofcurv}) holds with $k=1$ and $b \geq 2$;
    \item \label{lesserthantwo} $p \in (n+1, \frac{1}{2-b})$ if the curvature
    condition (\ref{deacyofcurv}) holds with $a >0$ and the curvature condition
    (\ref{deacyofderivativeofcurv}) holds with $k=1$ and $2 - \frac 1{n+1} < b < 2$.
\end{enumerate}
Moreover the harmonic coordinate functions $y^\gamma$ are all even
with respect to the variable $\tau$.
\end{lemm}

\begin{proof}
Recall, in local coordinate $\phi$
$$
\bar g =d\tau^2 + \bar g_{ij}d \theta^{i} d \theta^{j}, ~~~1\leq
i,j\leq n.
$$
Let
$$
\delta_{h} y =\frac{y(\tau+h, \theta)-y(\tau, \theta)}{h},
$$
for $(\tau,\theta^1, \cdots, \theta^n)\in D_{\frac{3}{2}\tau_1}$ and
$|h|<\frac{\tau_{1}}{4}$. Then
$$
\delta_{h} \bar \Delta y=0,
$$
in $D_{\frac{3}{2}\tau_1}$. From the fact that
$$
\|y\|_{W^{2,p}(D_{2\tau_1})}\leq C,
$$
and
$$
\frac{\p^{2}\bar{g}_{ij}}{\p\tau\p x}\in L^{p},
$$
due to (\ref{2mixderiestimate}) in Proposition \ref{basicestimate}
when $(b - 2)p > -1$, we get
$$
\| \bar g^{ij}\frac{\p^2}{\p x^i \p x^j} \delta_{h} y \|_{L^p
(D_{\frac{3}{2}\tau_0})}\leq C,
$$
which implies
$$
\| \delta_{h} y\|_{W^{2,p}(D_{\frac{3}{2}\tau_1})}\leq C.
$$
Let $h$ tend to zero. Then the Lemma follows from the Sobolev
embedding theorem. And evenness of each harmonic coordinate function
follows simply from the maximum principle.
\end{proof}

As a corollary, we have

\begin{coro}\label{ptauw1p}
Suppose that the curvature condition (\ref{deacyofcurv}) holds with
$a >0$ and the curvature condition (\ref{deacyofderivativeofcurv})
holds with $k=1$ and $b > 2 -\frac 1{n+1}$.  And let $\hat
g_{\alpha\beta}$ be the components of metric $\bar g$ in the above
harmonic coordinates. Then $\frac{\p}{\p \tau}\hat
g_{\alpha\beta}\in W^{1, p}(\mathcal{U})$ for $p\in
(n+1,\frac{1}{2-b})$, where $\mathcal{U}\subseteq
D_{\frac{3}{2}\tau_{1}}$. Moreover $\hat g_{\alpha\beta}$ is even
with respect to $\tau$ and
$$
\frac{\p}{\p \tau}\hat g_{\alpha\beta} = 0 \quad\text{at $\tau =0$}.
$$
\end{coro}

As we doubled the manifold and extend the metric to the doubled
manifold $N$ evenly, we have

\begin{equation}
\bar Rm(\tau, \theta)=\left\{
\begin{array}{ll}
    \bar Rm(\tau, \theta), \quad \tau >0, \\
    \bar Rm(-\tau, \theta), \quad \tau < 0.  \\
\end{array}
\right.
\end{equation}
To finally utilize the curvature equations we need
\begin{prop}\label{weakcurvature}

Suppose that the curvature condition (\ref{deacyofcurv}) holds with
$a >0$ and the curvature condition (\ref{deacyofderivativeofcurv})
holds with $k=1$ and $b > 2 -\frac 1{n+1}$. Then $\hat
g_{\alpha\beta}$ is weak solution to the following equations
$$
\frac{1}{2}\bar \Delta \hat g_{\alpha\beta} + Q_{\alpha\beta}(\p
\hat g,\hat g)=-\hat R_{\alpha\beta},
$$
in $\mathcal{U}$, where $\bar\nabla$ is gradient operator with
respect to metric $\bar g$, $Q(\p \hat g,\hat g)$ is bilinear form
of $\hat g$ and its first derivative, $d\bar V$ is the volume form
with respect to $\bar g$, $\hat R_{\alpha\beta}$ is component of
Ricci curvature tensor in $(\mathcal{U}, y^\gamma)$ excluding the
boundary $\tau=0$.
\end{prop}

\begin{proof}
Let
$$
\mathcal{U}_+ =\{(\tau, \theta)\in \mathcal{U}|\tau\geq 0\},
$$
and
$$
\mathcal{U}_- =\{(\tau, \theta)\in \mathcal{U}|\tau\leq 0\},
$$
$$
T=\mathcal{U}_+ \bigcap \mathcal{U}_-,
$$
$\vec{n}_\pm$ be the outward unit normal vector of $\mathcal{U}_\pm$
on $T$. Since $\hat g $ is smooth in $\mathcal{U}$ except on $T$, we
have, for each smooth function supported inside $\mathcal{U}$,
$$
\frac{1}{2}(\int_T \frac{\p \hat g_{\alpha\beta}}{\p
\vec{n}_+}\cdot\eta d\bar S-\int_{\mathcal{U}_+}\bar \nabla \hat
g_{\alpha\beta}\cdot\bar \nabla \eta d\bar V)+\int_{\mathcal{U}_+}
Q_{\alpha\beta}(\p \hat g,\hat g)\cdot \eta d\bar V
=-\int_{\mathcal{U}_+} \hat R_{\alpha\beta}\cdot\eta d\bar V,
$$
and
$$
\frac{1}{2}(\int_T \frac{\p \hat g_{\alpha\beta}}{\p
\vec{n}_-}\cdot\eta d\bar S-\int_{\mathcal{U}_-}\bar \nabla \hat
g_{\alpha\beta}\cdot\bar \nabla \eta d\bar V)+\int_{\mathcal{U}_-}
Q_{\alpha\beta}(\p \hat g,\hat g)\cdot \eta d\bar V
=-\int_{\mathcal{U}_-} \hat R_{\alpha\beta}\cdot\eta d\bar V.
$$
Here we used Proposition \ref{estimatesofcompactifiedcurvature} and
Corollary \ref{ptauw1p}. Note that on $T$,
$$
\frac{\p }{\p \vec{n}_-}=\frac{\p}{\p \tau}, \quad \frac{\p }{\p
\vec{n}_+}=-\frac{\p}{\p \tau}.
$$
By Corollary \ref{ptauw1p}, we hence obtain
$$
\int_T\frac{\p \hat g_{\alpha\beta}}{\p \vec{n}_+}\cdot\eta d\bar
S=\int_T\frac{\p \hat g_{\alpha\beta}}{\p \vec{n}_-}\cdot\eta d\bar
S=0,
$$
which completes the proof of the Proposition.
\end{proof}

Now we are ready to use harmonic coordinates at the infinity to
prove a regularity result as a step stone for the proof of Theorem
\ref{mainresult3}..

\begin{theo}\label{mainresult1} Suppose that
$(X^{n+1}, \ g)$ is a complete Riemannain manifold with an essential
set $\mathbf{E}$ and that it satisfies the curvature condition
(\ref{deacyofcurv}) with $a > 0$ and the curvature condition
(\ref{deacyofderivativeofcurv}) with $k=1$ and $b > 2 - \frac
1{n+1}$. Then there is differentiable structure on $\bar X$, which
is smooth in the interior of $X$, and $W^{3,p}$ up to the boundary,
where $p$ is some positive constant bigger than $n+1$. And in this
differentiable structure,  we have
\begin{enumerate}
    \item   $\bar g$ is smooth in the
    interior of $X$ and is $W^{2,p}$ up to the
    boundary for some $p>n+1$. In particular, it is $C^{1, \alpha}$
    smooth up to the boundary, for some $\alpha \in (0,1)$;
    \item $\bar g$ is smooth in the interior of $X$ and it is $W^{2,p}$ up to the
    boundary for any $p>n+1$ if in fact $b\geq 2$. In particular, it is
    $C^{1, \alpha}$ smooth up to the boundary, for any $\alpha \in (0,1)$.
\end{enumerate}
\end{theo}

\begin{proof}[Proof of Theorem \ref{mainresult1}]
By the definition and a direct computation, we see that $\hat
g_{\alpha\beta}\in{W^{1,p}(\mathcal{U},y^\gamma)}$, for some $p\in
(n+1, \frac 1{2-b})$, Then, by Proposition \ref{weakcurvature} and
the standard $L^p$ theory in PDE, we obtain that $\hat
g_{\alpha\beta}\in $ $W^{2,p}(\mathcal{U},y^\gamma)$ and there is a
constant $C$ which depends on $p$ such that
$$
 \|\hat g_{\alpha\beta}\|_{L^p (\mathcal{U})} +  \sum_\gamma
\|\frac{\p }{\p y^\gamma} \hat g_{\alpha\beta}\|_{L^p (\mathcal{U})}
+ \sum_{\gamma, \mu}\|\frac{\p^2 \hat g_{\alpha\beta}}{\p y^\gamma
\p y^\mu} \|_{L^p (\mathcal{U})} \leq C.
$$
Particularly $\hat{g}_{\alpha\beta}\in
C^{1,\mu}(\mathcal{U},y^\gamma)$ due to the standard Sobolev
embedding theorem for some $\mu \in (0, 1)$.

Now we have constructed the harmonic coordinates
$(\mathcal{U},y^\gamma)$ on $N$, and $\hat g_{\alpha\beta}$ has
better regularity than $\bar g_{\alpha\beta}$. By taking
$\mathcal{U}\bigcap \bar M$, which is still denoted by
$\mathcal{U}$, we get a coordinates covering on $M$. Let
$(\mathcal{U},y^\gamma)$ and $(\mathcal{V},z^\sigma)$ be two
distinct harmonic coordinates on $(N, \ \bar g)$. And suppose
$\mathcal{U}\bigcap\mathcal{V}$ is nonempty. Then, by the standard
arguments in PDE, we see that there is a constant $C$ with
$$
\| z^\sigma \|_{W^{3,p}(\mathcal{U}\bigcap\mathcal{V},y^\gamma)}\leq
C
$$
Putting all these harmonic coordinates together we then obtain a
differentiable structure on $\bar X^{n+1}$. Thus we finish to prove
Theorem \ref{mainresult1}.
\end{proof}

\begin{proof}[Proof of Theorem \ref{mainresult3}]
Recall from Proposition \ref{estimatesofcompactifiedcurvature}
$$
\|\bar \nabla Rm(\bar g)\|_{\bar g} \leq \Lambda \tau^{-(3-b)}.
$$
Hence by Lemma \ref{elementarylemma} and $b >2$, we know
$\bar{R}m\in C^{0,\mu_{1}}(\mathcal{U})$.  On the other hand, $\p
\hat{g}_{\alpha\beta}\in C^{0,\mu_{2}}(\mathcal{U})$, due to Theorem
\ref{mainresult1}. Therefore we may recall the equation
$$
\frac{1}{2}\bar \Delta \hat g_{\alpha\beta} + Q_{\alpha\beta}(\p
\hat g,\hat g)=-\hat R_{\alpha\beta},
$$
where $Q_{\alpha\beta}(\p \hat g,\hat g)$ and $\hat R_{\alpha\beta}$
are all H\"{o}lder continuous. Thus Theorem \ref{mainresult3}
follows from the standard Schauder theory in elliptic PDE.
\end{proof}

\section {Improvement by the Ricci flow }

In this section we show that the Ricci flow can be used to help to
get the regularity without assuming the curvature condition
(\ref{deacyofderivativeofcurv}). Let $(X^{n+1}, \ g)$ be a complete
Riemannian manifold satisfying the curvature condition
(\ref{deacyofcurv}) with $a > 0$. We consider the Ricci flow
equations
\begin{equation}\label{unnormalizedequation}
\left\{
\begin{array}{ll}
    \frac{\partial}{\partial t}g_{\alpha\beta} = -2R_{\alpha\beta}  \\
    g_{\alpha\beta}(x,0)=g_{\alpha\beta}(x) \quad \text{on}\quad X. \\
\end{array}
\right.
\end{equation}
We will also consider the normalized Ricci flow as follows:
\begin{equation}\label{normalizedequation}
\left\{
\begin{array}{ll}
    \frac{\partial}{\partial t}g_{\alpha\beta} = -2ng_{\alpha\beta}-2R_{\alpha\beta},  \\
    g_{\alpha\beta}(x,0)=g_{\alpha\beta}(x) \quad \text{on}\quad X,\\
\end{array}
\right.
\end{equation}
We recall that, if $g^N$ solves the normalized Ricci flow
(\ref{normalizedequation}) and let
\begin{equation}\label{scaling}
\left\{
\begin{array}{ll}
    \tau (t)=\frac{1}{2n}e^{2nt}-\frac{1}{2n},  \\
    g^U_{\alpha\beta}(x,\tau)=(1+2n\tau)g^N_{\alpha\beta}(x,t),\\
\end{array}
\right.
\end{equation}
then $g^U$ solves the Ricci flow (\ref{unnormalizedequation}). Hence
 if we know one we know the other. Since our initial metric
satisfies the curvature condition (\ref{deacyofcurv}), by the works
of Shi (see Theorem 1.1 in \cite{Shi}), there exist constants
$T,C_1,C_2>0$, which only depend on $n$ and the constant in
(\ref{deacyofcurv}) such that the evolution equation
$(\ref{unnormalizedequation})$ has a smooth solution $g(\cdot, t)$
for a short time $t\in [0, T]$, which satisfies the following
estimates:
$$
\sup_{x\in X}\|R_{\alpha\beta\gamma\delta}(x,t)\|\leq\ C_{1},~~0\leq
t\leq T
$$
and
$$
\sup_{x\in X}\|\nabla
R_{\alpha\beta\gamma\delta}(x,t)\|\leq\frac{C_{2}}{\sqrt{t}
},~~0\leq t\leq T.
$$
Let
$$
E_{\alpha\beta\gamma\delta}(g^U):=R_{\alpha\beta\gamma\delta}(g^U)
+\frac{1}{1+2(n-1)t}(g^U_{\alpha\gamma}g^U_{\beta\delta}-g^U_{\alpha\delta}g^U_{\beta\gamma}).
$$
Then
$$
E_{\alpha\beta\gamma\delta}(g^N) = R_{\alpha\beta\gamma\delta}(g^N)
+
(g^N_{\alpha\gamma}g^N_{\beta\delta}-g^N_{\alpha\delta}g^N_{\beta\gamma}).
$$
Notice that $\nabla E = \nabla Rm$. Direct computations from the
evolution equations of Riemann curvature tensor give us the
following:

\begin{lemm}\label{evolutionofesquare} Suppose that $g$ solve the
Ricci flow (\ref{unnormalizedequation}) and that $E = E(g)$. Then
the evolution equations of $\|E\|^{2}$ and $\|\nabla E\|^{2}$ are
given by
$$
\frac{\partial}{\partial t}\|E\|^{2}=\Delta \|E\|^{2}-2\|\nabla
E\|^{2}+E\ast E\ast E+E\ast E
$$
and
$$
\frac{\partial}{\partial t}\|\nabla E\|^{2}
 =\Delta \|\nabla E\|^{2}-2\|\nabla^{2}E\|^{2}+\nabla E\ast\nabla
E+E\ast\nabla E\ast\nabla E,
$$
where $E*E$ stands for terms that are contractions of $E$ and $E$.
\end{lemm}

Now we will use the maximum principle to estimate $\|E\|$ and
$\|\nabla E\|$.
\begin{lemm}\label{laplacianofinitialrho}
Let $g(t)$ be the smooth solution of (\ref{unnormalizedequation})
and $\rho_{t}$ the distance function corresponding to $g(t)$, $0\leq
t\leq T$, then $$|\Delta_{g(t)}\rho_{0}|\leq C, ~~~~0\leq t\leq T,$$
here $C$ depends only on $n$ and the constant in
(\ref{deacyofcurv}).
\end{lemm}

\begin{proof} First we recall a fact that
\begin{equation}\label{evolutionofxtofel}
\frac{\partial \Gamma^\gamma _{\alpha\beta}}{\partial t}=
-g^{\gamma\delta}(R_{\alpha\delta,\beta}+
R_{\beta\delta,\alpha}-R_{\alpha\beta,\delta}).
\end{equation}
Then we calculate
\begin{equation}
\begin{split}
\frac{\partial}{\partial t}(\Delta_{g(t)}\rho_{0}) & =
\frac{\partial}{\partial t}(g^{\alpha\beta}(\nabla^2\rho_0)_{\alpha\beta})\\
&
=2g^{\alpha\gamma}g^{\beta\delta}R_{\gamma\delta}(\nabla^2\rho_0)_{\alpha\beta}
+g^{\alpha\beta}g^{\gamma\delta}(R_{\beta\delta,\alpha}+R_{\alpha\delta,\beta}-
R_{\alpha\beta,\delta})\partial_{\gamma}\rho_{0}.
\end{split}\nonumber
\end{equation}
Due to Bianchi identity
$$
g^{\alpha\beta}(R_{\beta\delta,\alpha}+R_{\alpha\delta,\beta}-R_{\alpha\beta,\delta})
= 0
$$
we arrive at
\begin{equation}
\begin{split}
\frac{\partial}{\partial t} & (\Delta_{g(t)}\rho_{0})
=2g^{\alpha\gamma}g^{\beta\delta}R_{\gamma\delta}(\nabla^2_0\rho_0)_{\alpha\beta}\\
= &
2g^{\alpha\gamma}g^{\beta\delta}R_{\gamma\delta}(\nabla^2_0\rho_0)_{\alpha\beta}+
2g^{\alpha\gamma}g^{\beta\delta}
g^{\gamma\sigma}R_{\gamma\delta}(R_{\beta\sigma,\alpha}+R_{\alpha\sigma,\beta}
-R_{\alpha\beta,\sigma})(s)\partial_{\gamma}\rho_{0}\cdot
t,\\
\end{split}\nonumber
\end{equation}
where $0<s<t\leq T$, $\nabla^2_0\rho_0$ is Hessian of $\rho_0$ with
respect to metric $g_0$. This lemma then follows from Shi's work
\cite{Shi}.
\end{proof}

Now let $\mu$, $\nu$ and $\eta$ be constants which are to be
determined later. By we calculate from Lemma
\ref{laplacianofinitialrho} that
$$
\aligned \frac{\partial}{\partial t}(e^{\mu\rho_{0}}\|E\|^{2}) &
\leq \Delta(e^{\mu\rho_{0}}\|E\|^{2})
\\ +<2\mu\nabla\rho_{0},\nabla(e^{\mu\rho_{0}}\|E\|^{2})> & + C(e^{\mu\rho_{0}}\|E\|^{2})-2e^{\mu\rho_{0}}\|\nabla
 E\|^{2},\endaligned
$$
and
\begin{equation}
\begin{split}
\frac{\partial}{\partial t}(e^{\nu\rho_{0}}\|\nabla E\|^{2})&
\leq\Delta(e^{\nu\rho_{0}}\|\nabla
E\|^{2})+<2\nu\nabla\rho_{0},\nabla(e^{\nu\rho_{0}}\|\nabla
E\|^{2})>\\
&+C(e^{\nu\rho_{0}}\|\nabla E\|^{2}) -2e^{\nu\rho_{0}}\|\nabla^{2}
 E\|^{2},
\end{split}\nonumber
\end{equation}
where $C$ is large enough and independent of $t$. Therefore
\begin{equation}
\begin{split}
\frac{\partial}{\partial t}(e^{\mu\rho_{0}}\|E\|^{2}+\eta
te^{\nu\rho_{0}}\|\nabla E\|^{2}) &
\leq\Delta(e^{\mu\rho_{0}}\|E\|^{2}+\eta te^{\nu\rho_{0}}\|\nabla
E\|^{2})\\
 +C(e^{\mu\rho_{0}}\|E\|^{2}+\eta
te^{\nu\rho_{0}} & \|\nabla E\|^{2})+
(\eta e^{\nu\rho_{0}} -2e^{\mu\rho_{0}})\|\nabla E\|^{2}\\
& \leq\Delta(e^{\mu\rho_{0}}\|E\|^{2}+\eta te^{\nu\rho_{0}}\|\nabla
E\|^{2})\\
&+C(e^{\mu\rho_{0}}\|E\|^{2}+\eta te^{\nu\rho_{0}}\|\nabla E\|^{2}),
\end{split}\
\end{equation}
if we choose constants $\mu=\nu$ and $\eta\leq 2$.\\

Thus we have:

\begin{prop}\label{curvaturedecayestimate}
Suppose $(X^{n+1}, \ g)$ is a complete Riemannian manifold
satisfying the curvature condition (\ref{deacyofcurv}) with $a > 0$.
Let $(M,g(t))$, $t\in [0,T],$ be a complete solution of the
normalized Ricci flow (\ref{normalizedequation}). Then there exists
constants $T_{0}$ and $C$ such that
$$
\|E(g(t))\|\leq Ce^{-a\rho_{0}}
$$
and
$$
\|\nabla E(g(t))\|\leq \frac{C}{\sqrt{t}}e^{-a\rho_{0}},$$ where
$\nabla$ is with respect to metric $g(t)$, and $C$ is independent of
$t$, $0<t\leq$ $T_{0}\leq T.$
\end{prop}

\begin{proof}
It suffices to prove the same conclusion for unnormalize Ricci flow
(\ref{unnormalizedequation}). From the analysis above, choose $T$
smaller if necessary, we know $(X, \ g(t))$, $t$ $\in [0,T],$ be a
complete solution of the Ricci flow (\ref{unnormalizedequation})
with uniformly bounded curvature. In fact, due to
(\ref{deacyofcurv}) on $(X, \ g(0))$,
$$
e^{2a\rho_{0}}\|E\|^{2}(x,0)\leq C.
$$
So $u(x,t)=e^{2a\rho_{0}}\|E\|^{2}+ te^{2a\rho_{0}}\|\nabla
E\|^{2}-C$ is a weak subsolution of the heat equation
$(\frac{\partial}{\partial t}-\bigtriangleup)u(x,t)=0$ on
$X\times[0,T]$ with $u(x,0)=e^{2a\rho_{0}}\|E\|^{2}-C\leq0$. On the
other hand, we can choose sufficiently large $\omega$ such that
$$
\int_{0}^{T_{1}}\int_{M}exp(-\omega
d_{g_0}^{2}(x,o))u_{+}^{2}(x,s)dV_{g(t)}dt<\infty,
$$
where $d_{g_0}(x,o)$ is a distance function to a fixed point $o\in
M$ with respect to $g_0.$ Then according to the results of Krap and
Li \cite{KL}, $u(x,t)\leq 0$ for all $(x,t)\in X\times [0,T]$. Then
$$
e^{2a\rho_{0}}\|E\|^{2}-C\leq0
$$
and
$$
\eta te^{\mu\rho_{0}}\|\nabla E\|^{2}-C\leq0.
$$
Hence the proof is complete.
\end{proof}

So far the normalized Ricci flow $g(x,t)$ preserves the asymptotical
curvature behavior of the initial metric. Next we want to show that
there is a compact set $\mathbf{E}\subset X$ such that it is
essential set for all $g(x, t)$, $t\in [0, T]$, if
$\mathbf{E}\subset X$ is an essential set for the initial metric, at
least for a sufficiently small $T$. To do that, according to Theorem
$4$ in \cite{BM}, it suffices to show that $\mathbf{E}$ is strictly
convex. Since, in the light of Proposition
\ref{curvaturedecayestimate}, we may choose $\mathbf{E}$ large
enough so that sectional curvature of $g(x, t)$ is negative out of
$\mathbf{E}$ for $t\in[0, T]$.

\begin{prop}\label{derivativeestimate1}  Suppose $(X^{n+1}, \ g)$
is a complete Riemannian manifold with an essential set $\mathbf{E}$
and that it satisfies the curvature condition (\ref{deacyofcurv})
with $a > 0$. Let $\rho$ be the distance to $\partial\mathbf{E}$ in
$(X^{n+1}, \ g)$. And suppose that $g(\cdot, t)$ is the solution to
the normalized Ricci flow equations (\ref{normalizedequation}). Then
there is $T>0$ and $\Lambda_0$ such that
$$
\mathbf{E}_1 = \{x\in X: \rho \leq 2\Lambda_0\}
$$
is strictly convex with respect to $g(\cdot, t)$ for all $t\in [0,
T]$. Therefore $\mathbf{E}_1$ is an essential set in $X$ for all
$g(\cdot, t)$ with $t\in [0, T]$.
\end{prop}

\begin{proof} First we claim that there is $T>0$ and $\Lambda_0$
such that
$$
\nabla^{2}_{g(\cdot, t)}\cosh\rho \geq \frac{\sinh\rho}{4}g(\cdot,
t),
$$
for all $t\in [0, T]$ and $\rho \geq \Lambda_{0}$. Let
$\frac{\partial}{\partial \theta^0}=\frac{\partial}{\partial \rho}$.
Then
$$
\nabla^{2}_{g_{0}}\cosh\rho (\frac{\partial}{\partial
\rho},\frac{\partial}{\partial \rho})=\cosh\rho,
$$
$$
\nabla^{2}_{g_{0}}\cosh\rho (\frac{\partial}{\partial
\rho},\frac{\partial}{\partial \theta ^i})=0,
$$
and
$$
\nabla^{2}_{g_{0}}\cosh\rho (\frac{\partial}{\partial \theta
^i},\frac{\partial}{\partial \theta ^j})=-\Gamma_{ij}^{1}\sinh\rho
=S_{i}^{k}g_{kj}(\rho,\theta)\sinh\rho.
$$
We then can choose $\Lambda _{0}$, which is independent of $t$, such
that for all $\rho \geq \Lambda _{0}$,
$$
S_{i}^{k}g_{kj}(\rho,\theta)\sinh\rho \geq \frac{1}{2}\sinh\rho
\cdot g_{ij}(\rho,\theta),
$$
in the light of (\ref{estimatesp1}), (\ref{estimatesp2}) and
(\ref{estimatesp3}). Therefore
$$
\nabla^{2}_{g_{0}}\cosh\rho \geq \frac{1}{2}\sinh\rho \cdot g_{0}
$$
At $t=0$. On the other hand,
\begin{equation}
\begin{split}
\frac{\partial}{\partial t}(\nabla_{i}\nabla_{j}\cosh\rho)
&= \sinh\rho g^{0l}(R_{jl,i}+R_{il,j}-R_{ij,l})\\
&= \sinh\rho(R_{j0,i}+R_{i0,j}-R_{ij,0}),\\
\end{split}\nonumber
\end{equation}
which implies
$$
|\frac{\partial}{\partial t}(\nabla_{i}\nabla_{j}\cosh\rho)|\leq
C\sinh\rho ||\nabla Ric||\leq
Ct^{-\frac{1}{2}}e^{-\alpha\rho_{0}}\sinh\rho.
$$
Choose $\Lambda_{0}$ large enough and time interval small enough if
necessary, we hence proved the first claim.

To see that $\mathbf{E}_1$ is strictly convex in all $(X, g(t))$.
Let $f=\cosh\rho$. For any $p$, $q\in \mathbf{E}_1$, then $f(p)\leq
\cosh(2\Lambda_{0}),~f(q)\leq \cosh(2\Lambda_{0})$. Let $\gamma_t$
be any geodesic joining $p$ and $q$ in the metric $g(\cdot, t)$.
Then
$$
f(\gamma_t(s))^{''} > 0, \quad \text{for all $s\in (0, 1)$},
$$
from which it is easily seen that
$$
f(\gamma_t(s)) < \cosh(2\Lambda_0),
$$
for all $s\in (0, 1)$. Thus we construct an essential set
$\mathbf{E}$ of $(M, \ g(t))$, for all $t\in [0,T]$.
\end{proof}

By now we may use the regularity results in the previous section and
regularity theorem of \cite{BG} to the metric $g(\cdot, t)$ for
$t\in (0, T]$. First we notice that

\begin{prop}\label{convergencethem} Suppose $(X^{n+1}, \ g)$ is a complete
Riemannian manifold with an essential set $\mathbf{E}$ and that it
satisfies the curvature condition (\ref{deacyofcurv}) with $a > 0$.
Let $g(\cdot, t)$ be the solution to the normalized Ricci flow
$(\ref{normalizedequation})$ and that $\mathbf{E}$ is an essential
set for all $g(\cdot, t)$. Let $\bar g_{ij}(\cdot,
t)=\sinh^{-2}\rho_t g(\cdot, t)$, where $\rho_t$ is the distance
function to $\mathbf{E}$ with respect to the metric $g(\cdot, t)$.
Then there are constants $T$ and $C$ such that
$$
\|\bar{g}(\cdot,t)-\bar{g}(\cdot, 0)\|\leq Ct
$$
for all $t\in [0, T]$.
\end{prop}

\begin{proof} According to the equations
$$
\frac{\partial}{\partial
t}g_{\alpha\beta}=-2ng_{\alpha\beta}-2R_{\alpha\beta},
$$
we easily see that
$$
\|g(\cdot, t)-g(\cdot, 0)\| \leq Ct,
$$
where $C$ is independent of $t$.

Let $\rho(\cdot, t)$ be the distance function to $B =
\partial\mathbf{E}$ with respect to $g(\cdot, t)$. For any $x\in
M\backslash \mathbf{E}$, suppose that $p\in \mathbf{B}$ is the point
such that the length of the geodesic joining $x$ and $p$ is just the
distance of $x$ to $\mathbf{B}$ with respect to $g(\cdot, 0)$.
Denote this geodesic by $\gamma(s)$, where $s\in [0,1],~\gamma(0)=
p, ~\gamma(1)=x, ~\rho(x,0)$ $ = L(\gamma,0)$ and $L(\gamma,t)$
denotes the length of $\gamma$ with respect to $g(\cdot, t)$. We
know that $g(\cdot, t)$ and $g(\cdot, 0)$ are uniformly
quasi-isometric when $t\in [0,T]$. Hence there exists a constant
$\lambda$, independent of $t$, such that
$$
\lambda^{-1}L(\gamma,0)\leq L(\gamma,t)\leq\lambda
L(\gamma,0),~\text {for any}~ t\in [0,T].
$$
On the other hand we may compute
\begin{equation}
\begin{split}
\frac{\partial}{\partial t}L(\gamma,t)&=\frac{\partial}{\partial
t}\int_{0}^{1}(g_{\alpha\beta}(\gamma,t)\dot{\gamma}^{\alpha}\dot{\gamma}^{\beta})^{\frac{1}{2}}ds\\
&=\int^{1}_{0}(g_{\alpha\beta}(\gamma,t)\dot{\gamma}^{\alpha}\dot{\gamma}^{\beta})^{
-\frac{1}{2}}(-ng_{\mu\nu}(\gamma,t)-R_{\mu\nu}(\gamma,t))
\dot{\gamma}^{\mu}\dot{\gamma}^{\nu}ds\\
&=-\int^{1}_{0}(g_{\alpha\beta}(\gamma,t)\dot{\gamma}^{\alpha}\dot{\gamma}^{\beta})^{
-\frac{1}{2}}E_{\mu\nu}(\gamma,t)\dot{\gamma}^{\mu}\dot{\gamma}^{\nu}ds,
\end{split}\nonumber
\end{equation}
where
$E_{\alpha\beta}=g^{\gamma\delta}(t)E_{\alpha\gamma\beta\delta}(g(\cdot,
t))$, and
$$
E_{\alpha\beta\gamma\delta}(g) = R_{\alpha\beta\gamma\delta}(g) +
(g_{\alpha\gamma}g_{\beta\delta}-g_{\alpha\delta}g_{\beta\gamma}).
$$
Hence according to Proposition \ref{curvaturedecayestimate} we have
\begin{equation}
\begin{split}
\int^1_0(g_{ij}(\gamma,t)\dot\gamma^i\dot\gamma^j)^{-\frac12}E_{ij}(\gamma,t)\dot\gamma^i\dot\gamma^jds
&\leq \int^1_0 Ce^{-\alpha
sL_{0}}\lambda^{\frac{1}{2}}(g_{ij}(\gamma,0)\dot{\gamma}^{i}\dot\gamma^{j})^{\frac{1}{2}}ds\\
&\leq C\lambda^{\frac{1}{2}}L_{0}\int^{1}_{0}e^{-a s L_0 }ds\\
&\leq C\lambda^{\frac{1}{2}}a^{-1},
\end{split}\nonumber
\end{equation}
which implies
$$
|L(\gamma,t)-L_{0}|\leq\frac{C\lambda^\frac12}{a}t,~\text {for any}~
t\in [0,T],
$$
and
$$
\rho(x,t)\leq L(\gamma,t)\leq
L_{0}+\frac{C\lambda^\frac12}{a}t=\rho(x,0)+\frac{C\lambda^\frac12}{a}t.
$$
Similarly we get
$$
\rho(x,0)=L_{0}\leq \rho(x,t)+\frac{C\lambda^\frac12}{a}t.
$$
Therefore
$$
|\rho(x,t)-\rho(x,0)|\leq\frac{C\lambda^\frac12}{a}t,
$$
for any $t\in [0, T]$. Thus
$$
\|\bar{g}(x,t)-\bar{g}(x,0)\|\leq C(\lambda,a)t,
$$
for all $x\in M\setminus \mathbf{E}$, and all $t\in [0,T]$, which
completes the proof of the proposition.
\end{proof}

As another easy consequence of the use of the help from the Ricci
flow we now give the proof of Theorem \ref{holdercontinuous}.

\begin{proof}[Proof of Theorem \ref{holdercontinuous}]
Let $g(x,t)$ be the metric constructed by the normalized Ricci flow
(\ref{normalizedequation}). Then, according to Theorem A in
\cite{BG} and Proposition \ref{curvaturedecayestimate} in the above,
we have
$$
\|\bar {g}(x,t) \|_{C^{0, a}} \leq \Lambda t^{-\frac12},
$$
on $\bar X$. Recall from the proof of Proposition
\ref{convergencethem}
$$
|\bar g(x,t)-\bar g(x, 0)|\leq C t.
$$
Hence, if let $\mu =\frac23 a$, we have, for any $t>0$,
\begin{equation}
\begin{split}
\frac{\|\bar g(x, 0)-\bar g(y, 0)\|}{|x-y|^\mu}&\leq \frac{\|\bar
g(x,t)-\bar g(x, 0)\|}{|x-y|^\mu}+\frac{\|\bar g(y,t)-\bar
g(y, 0)\|}{|x-y|^\mu}\\
&+\frac{\|\bar g(x,t)-\bar g(y,t)\|}{|x-y|^\mu}\\
&\leq C t\cdot |x-y|^{-\mu}+\Lambda t^{-\frac12}\cdot |x-y|^{a-\mu}
\end{split}
\end{equation}
Take $t=|x-y|^{\frac23 a}$, we obtain
$$
\frac{\|\bar g(x, 0)-\bar g(y, 0)\|}{|x-y|^\mu}\leq C.
$$
Thus we finish the proof.
\end{proof}

Next we would like to drop the curvature condition
(\ref{deacyofderivativeofcurv}) in Theorem \ref{mainresult1} with
the help from the Ricci flow method indicated in the above. By
Theorem \ref{mainresult1}, if we denote $\{x^\gamma\}$ to be
harmonic coordinates with respect to the metric $\bar g(\cdot, t)$,
then $\bar{g}_{\alpha\beta}(x, t)$ is smooth in the interior of $X$
and it is $W^{2,p}$ up to the boundary, in the differential
structure associated with the harmonic coordinates. For a technical
reason we first smoothen each $\bar g(\cdot, t)$ in order to allow
us to apply Theorem 5.4 in \cite{GP} (page 187). We continue to use
the notations in the proof of Theorem \ref{mainresult1} in the
previous section.

\begin{lemm}\label{convergencethem2} Suppose $\bar{g}_{\alpha\beta}(\cdot, t)$ is in
$W^{2,p}$ for some $p>\frac{n+1}{2}$ in a coordinate neighborhood
$\mathcal U$ in the doubled manifold $N$. Let
$\bar{g}^\epsilon_{\alpha\beta}(\cdot, t)$ be the
$\epsilon$-mollification of $\bar{g}_{\alpha\beta}(\cdot, t)$, i.e.
let $\mu(x)$ is a smooth function with compact support inside
$\mathcal U$ such that $\int \mu(x)dx=1$ and let
$\mu_{\epsilon}(x)=\frac{1}{\epsilon^{n+1}}\mu(\frac{x}{\epsilon})$,
then
$$
\bar{g}^{\epsilon}_{\alpha\beta}(x, t) =
\int_{M}\bar{g}_{\alpha\beta}(y, t)\mu_{\epsilon}(x-y)dy.
$$
Then we have
\begin{enumerate}
  \item $\bar{g}^{\epsilon}_{\alpha\beta}(\cdot, t)$ is $C^\infty$ smooth;
  \item $\|Rm(\bar{g}^\epsilon)\|_{L^p}\leq C$, where $C$ is independent of
  k;
  \item There is $\epsilon_k\rightarrow 0$ so that
$\bar{g}^{\epsilon_k}_{\alpha\beta}(\cdot, \epsilon_k)$ converges to
$\bar{g}_{\alpha\beta}(\cdot, 0)$ uniformly in any compact subset of
$\mathcal U$.
\end{enumerate}
\end{lemm}

\begin{proof} We only need to prove the last two statements. For
convenience we sometime drop the indices for the metrics if there is
no confusion. Note that $W^{1,p}\subset L^{2p}$ since
$p>\frac{n+1}{2}$, we hence have
$$
\|(Rm(\bar g^\epsilon))\|_{L^{p}}\leq C \|\bar g\|_{W^{2,p}}.
$$
Now let us prove the third statement. Clearly we have
$$
|\bar{g}^{\epsilon}_{\alpha\beta}(\cdot, t) -
\bar{g}_{\alpha\beta}(\cdot, 0)|  \leq
|\bar{g}^{\epsilon}_{\alpha\beta}(\cdot,
t)-\bar{g}_{\alpha\beta}(\cdot, t)| + |\bar{g}_{\alpha\beta}(\cdot,
t)-\bar{g}_{\alpha\beta}(\cdot, 0)|.
$$
For the second term, by Proposition \ref{convergencethem}, we know
$$
\lim_{t\to 0}\|\bar{g}(\cdot, t)-\bar{g}(\cdot, 0))\|=0.
$$
As for the first term, since $p > \frac {n+1}2$, we know $\bar g$ is
continuous due to the assumptions. Therefore it is rather easy to
see that
$$
\lim_{\epsilon\to 0}\|\bar g^\epsilon - \bar g\|= 0.
$$
Thus we finish to prove the Proposition.
\end{proof}

For simplicity we denote $\bar{g}^{\epsilon_k} (\cdot, \epsilon_k)$
as $\bar{g}^k$ in the following.

\begin{lemm}\label{uniformestimateofmetric} Suppose $(X^{n+1}, \ g)$
is a complete manifold with an essential set and that it satisfies
the curvature condition (\ref{deacyofcurv}) with $a > 2 - \frac
2{n+1}$. Let $\mathcal U$ be the coordinate neighborhood in the
doubled manifold where $\bar g^k$ was constructed in Lemma
\ref{convergencethem2}. Then for any point on $T$ it admits a
neighborhood $\mathcal{V}\subset \mathcal U$ so that
\begin{enumerate}
  \item  for each $\bar g^k$ there is harmonic coordinates $H_k =(y^0_k, y^1_k, \cdots, y^n_k)$ on
  $\mathcal{V}$;
  \item there is a positive constant $\delta_0$, which is independent of $k$, such that
  $\det (dH_k)\geq \delta_0>0$ on $\mathcal{V}$;
  \item $\|\hat{g}^{k}_{\alpha\beta}\|_{W^{2,p}(\mathcal{V})}\leq C$,
  for some $p > \frac{n+1}{2}$,
  and $\hat{g}^{k}_{\alpha\beta}$ is components of $\bar g^k$ under harmonic coordinates
  $(\mathcal{V}, y^\gamma_k)$, and $C$ is
  constant that is independent of $k$.
\item If $a \geq 2$, then (3) is true for any $p>1$.
\end{enumerate}

\end{lemm}
\begin{proof}
Indeed, due to Proposition \ref{estimatesofcompactifiedcurvature}
and Lemma \ref{convergencethem2}, we know that
$$
\|Rm(\bar{g}^k) \|_{L^p}\leq C.
$$
Since $\bar{g}^{k}$ converges to  $ \bar{g}$ uniformly, we see that
for any $q$ in $\bar X$ and $s\leq 1$, we have $Vol(B(q,s))\geq
\Lambda s^n$, where $B(q,s)$ is geodesic ball with radius $s$ and
center $q$ and $\Lambda$ is independent of $s$. Hence, by Theorem
5.4 in \cite{GP} (P.187), we obtain the existence of the
neighborhood $\mathcal V$ and therefore the Lemma is proven.
\end{proof}

Notice that we have proved that $\bar{g}^{k}_{\alpha\beta}$
converges to $\bar{g}_{\alpha\beta}$ uniformly and that
$$
\|\hat{g}^{k}_{\alpha\beta}\|_{W^{2,p}(\mathcal{V})}\leq C
$$
where C is independent of $k$ and $p > n+1$ when $a > 2 - \frac
1{n+1}$. Because each $y^{\gamma}_k$ is harmonic with respect to
metric $\bar g^k$, from the standard theory of elliptic PDE we have
$$
||y^{\gamma}_k||_{C^{2,\mu}(\mathcal{V})}\leq C,
$$
for some $\mu>0$. Now we are well prepared to prove
Theorem\ref{mainresult2}.

\begin{proof}[Proof of Theorem \ref{mainresult2}]
The a priori estimates from the standard theory of elliptic PDE are
always available, but the missing key point is to find a fixed
common domain $\mathcal V$ for the family of elliptic PDE's and
their solutions, which has been shown using Theorem 5.4 in \cite{GP}
(P.187) in the proof of Lemma \ref{uniformestimateofmetric}. Indeed,
by Theorem \ref{mainresult1} and Lemma
\ref{uniformestimateofmetric}, for each $k$, we have differential
structure $\Gamma_k =\{(\mathcal{V},y^\gamma_k)\}$ and the
components $\hat g^k_{\alpha\beta}$ for $\bar g^k$ in the harmonic
coordinates $\{y^\gamma_k\}$. Then, taking a subsequence if
necessary, we may assume
\begin{enumerate}
\item $y^\gamma_k$ converges to $y^\gamma$ in $C^{2, \mu'}(\mathcal{V'})$
      for some $\mu' < \mu$ and $\mathcal V' \subset\subset \mathcal V$
      and;
\item $\hat g^k_{\alpha\beta}$ converges weakly to $\hat g_{\alpha\beta}$
      in $W^{2,p}(\mathcal V)$,
\end{enumerate}
where clearly $\hat g_{\alpha\beta}$ is the components of the metric
$\bar g(\cdot, 0)$ in the coordinates $y^\gamma$ if $\{y^\gamma\}$
is indeed a coordinate system, which is readily seen because
$\{y_k^\gamma\}$ are coordinate systems.

Finally, suppose $(\mathcal{U},y^\gamma)$, $(\mathcal{V},z^\gamma)
\in \Gamma$, then $z^\gamma$ is harmonic in
$\mathcal{U}\bigcap\mathcal{V}$ with respect to metric $\bar g$, and
$\bar g$ is $W^{2,p}$-smooth under harmonic coordinates
$(\mathcal{U}\bigcap\mathcal{V}, $ $y^\gamma)$, thus, we get
$$
\| \frac{\p z^\gamma}{\p
y^\nu}\|_{W^{2,p}(\mathcal{U}\bigcap\mathcal{V})}\leq C,
$$
where $C$ is a constant depends only $\Gamma$ and $p$. Thus we
finish to prove the Theorem \ref{mainresult2}.
\end{proof}

\section {Rigidity Theorems}

About rigidity in this context there are three different approaches
given in \cite{AD}, \cite{Q1} and \cite{ST} (see also \cite{DJ})
respectively. In \cite{AD} the manifolds are assumed to be spin and
the regularity of the conformal compactification is assumed to be
very high. In \cite{Q1} it still assumes the regularity of order
$C^{3, \alpha}$, even though no spin condition is assumed for $n\leq
7$. In \cite{ST} it takes the advantage of the volume comparison of
geodesic spheres, hence it assumes the manifold to have a pole. Very
recently in \cite{DJ} the authors seemed to be able to relate the
conformal infinity and geodesic spheres and obtained a nice rigidity
theorem for asymptotically hyperbolic with $C^2$ conformal
compactifications. Hence it is easily seen that Theorem
\ref{mainresult3} becomes a significant step stone to utilize the
regularity theorem in \cite{CDLS} of the conformally compact
Einstein metrics to apply any available rigidity result. Let us
first state and prove a rigidity theorem as an easy consequence of
Theorem \ref{mainresult3}.

\begin{theo}\label{Einstein}
Suppose that $(X^{n+1}, \ g)$ is a complete Einstein manifold with
$Ric = -ng$. And suppose that it has an essential set $\mathbf E$
and that it satisfies the curvature condition (\ref{deacyofcurv})
with $a > 2$. Also Suppose that $X^{n+1}$ is simply connected at the
infinity. Then $(X^{n+1}, \ g)$ is a standard hyperbolic space if
$4\leq n\leq 6$ or $X^{n+1}$ is spin if $n\geq 7$.  And it is a
standard hyperbolic space if in addition we assume that $\int_X \|Rm
-\mathbf{K}\|d\mu_g <\infty$ if $n=3$.
\end{theo}

\begin{proof}
Because of (\ref{deacyofcurv}) and Einstein equations, applying
Theorem 4.3 in \cite{BG}, we see that
$$
\|\nabla Rm\| \leq C e^{- a \rho}.
$$
Then by Theorem \ref{mainresult3} proven in \S3, we see that $\bar
g$ is $C^{2,\mu}$ up to the infinity boundary for some $\mu\in (0,
1)$.

On the other hand, due to the curvature condition
(\ref{deacyofcurv}) with $a > 2$ (plus $\int_M
\|Rm-\mathbf{K}\|d\mu_g <\infty$ for $n=4$) and Theorem 2.6 in
\cite{ST}, we see that the conformal infinity $(M^n, [\hat g])$ of
$(X^{n+1}, \ g)$ is locally conformally flat. Then by simply
connectedness of $X$ at the infinity we see that the boundary $M^n$
of $X$ is simply connected. Hence the conformal infinity $(M^n,
[\hat g])$ is conformally equivalent to the standard sphere. As
$\bar g$ is $C^{2, \mu}$ smooth at the infinity boundary, we know
that there is a positive function $u\in C^{2, \mu}(M^n)$ so that
$u^{\frac{4}{n-3}}\bar g$ is the metric of the standard sphere, i.e.
$u$ satisfies
$$
\Delta_{\bar g}u- \frac{n-2}{4(n-1)} Ru + \frac 14 n(n-2)
u^{\frac{n+2}{n-2}}=0,
$$
By setting
$$
u(\rho,\theta)=u(\theta),
$$
we may assume $u$ is defined on $X$, hence $u^{\frac{4}{n-2}} \cdot
\bar g $ is $C^{2,\mu}$ smooth near the infinity boundary $M$ and
its restriction on $M$ is the standard sphere metric. Now due to
Theorem A in \cite{CDLS}, we know that in fact $(X^{n+1}, \ g)$ is
conformally compact of order $C^{\infty}$. Then the theorem follows
from the results in \cite{Q1} and \cite{AD}.
\end{proof}

To get rigidity results for asymptotically hyperbolic manifolds with
only $Ric \geq - ng$ we need to have the following technical lemma,
which is an improvement of Lemma 5.1 in \cite{jL}, based on an idea
from \cite{CDLS}. Namely,

\begin{lemm}\label{specialdefiningfunction} Suppose $(X^{n+1}, \ g)$
is a conformally compact manifold of regularity $C^{2, \lambda}$ for
$\lambda\in (0, 1)$. Then, for any $C^{2, \lambda}$ defining
function $r$ and $\bar g = r^2 g \in C^{2, \lambda}$, there is a
$C^{2, \lambda}$ geodesic defining function $x$ such that
$$
x^2 g = r^2 g = \bar g
$$
on the boundary $\p X$ and $x^2 g\in C^{2, \lambda}$.
\end{lemm}

\begin{proof} Let $\bar g = r^2 g$ and the coordinates near the boundary as
$(\theta^0, \theta^1, \cdots, \theta^n)$, where $\theta^0 = r$. To
find the geodesic defining function $x$ we set $x = e^u r$ and
consider the equations, as in \cite{jL},
$$
F(\theta, du) = 2 <du, dr>_{\bar g} + r|du|^2_{\bar g} - \frac {1 -
|dr|^2_{\bar g}}r = 0
$$
in a neighborhood of the boundary. We are using the method of
characteristics to solve this PDE, hence we turn to solve the system
of ODE
$$
\left\{
\begin{array}{lll}
\dot{p} = - D_\theta F (\theta(t), p(t)) \\
\dot{z} = D_pF(\theta(t), p(t)\cdot p(t) \\
\dot{\theta} = D_p F(\theta(t), p(t)),
\end{array}
\right.
$$
where $\theta(t)$ is the characteristic curve for the PDE,
$u(\theta(t)) = z(t)$ and $du(\theta(t))$ $= p(t)$. Readers are
refereed to the book \cite{cE} for the method of characteristics to
solve first order nonlinear PDEs. We know the somehow trouble term
is
$$
\frac {1 - |dr|^2_{\bar g}}r\in C^{1, \lambda},
$$
which was considered in the proof of Lemma 5.1 in \cite{jL} to be
responsible for the loss of regularity.

First the so-called non-characteristic nature of the PDE is meant
that
$$
D_{p_0}F(0, \theta^1, \cdots, \theta^n, p_0, 0, \cdots, 0) = 2 \neq
0
$$
at an admissible initial date set $(0, \theta^1, \cdot, \theta^n, 0,
p_0, 0, \cdots, 0)$ for $(\theta(0), z(0), p(0))$. But here it is
rather explicit that
$$
p_0 (\theta^1, \theta^2, \cdots, \theta^n) = \frac 12 \frac {1 -
|dr|^2_{\bar g}}r|_{r = 0} \in C^{2, \lambda}.
$$
We now consider $u$ as a function of variables $(t, p^1, p^2,
\cdots, p^n)$. We easily see that
$$
\p_t^2 u = \p_t \dot{z} = (D_\theta D_p F\cdot \dot{\theta})\cdot p
+ D_pF\cdot \dot{p}\in C^{0, \lambda}
$$
and
$$
\p_tdu = \dot{p} \in C^{0. \lambda}.
$$
It is then left only to verify that
$$
\p_{\theta^\alpha}\p_{\theta^\beta} u =
\p_{\theta^\alpha}\p_{\theta^\beta} z = \p_{\theta^\alpha}
p_\beta\in C^{0, \lambda},
$$
which equivalently is to verify that the solution of the system of
ODE smoothly depends on the initial data up to certain order. Here,
thanks to \cite{CDLS}, we take a change of variable that
$$
s = \log r.
$$
Then the function $F\in C^{2, \lambda}$ with respect to the
variables
$$
(s, \theta^1, \cdots, \theta^n, p_0, p_1, \cdots, p_n).
$$
(please see the proof of Lemma 6.1 in \cite{CDLS}) Therefore there
is no loss of regularity with respect to the variables $(\theta^1,
\theta^2, \cdots, \theta^n)$. Finally, note that $F\in C^{2,
\lambda}$, then by Implicit Function Theorem we see that $u\in C^{2,
\lambda}$ which implies $x$ is $C^{2, \lambda}$ too. Thus the proof
is finished.
\end{proof}

We are now ready to state and show a rigidity theorem for
asymptotically hyperbolic manifolds with $Ric \geq -n g$.

\begin{theo}\label{Riccicomparison}
Suppose that $(X^{n+1}, \ g)$ is a complete  manifold with $Ric \geq
-ng$. And suppose that it has an essential set $\mathbf E$ and that
it satisfies the curvature condition (\ref{deacyofcurv}) with $a >
0$ and (\ref{deacyofderivativeofcurv}) with $k=1$, $b>2$. Also
Suppose that $X^{n+1}$ is simply connected at the infinity. Then
$(X^{n+1}, \ g)$ is a standard hyperbolic space if $4\leq n\leq 6$
or $X^{n+1}$ is spin if $n\geq 7$.  And it is a standard hyperbolic
space if in addition we assume that $\int_X \|Rm -\mathbf{K}\|d\mu_g
<\infty$ if $n=3$.
\end{theo}

\begin{proof}
First we know from Theorem \ref{mainresult3} in \S3 that $\bar g$ is
$C^{2,\mu}$ up to the infinity boundary for some $\mu\in (0, 1)$.
And by the curvature condition (\ref{deacyofcurv}) with $a > 2$
(plus $\int_X \|Rm-\mathbf{K}\|d\mu_g <\infty$ for $n=3$) and
Theorem 2.6 in \cite{ST}, we see that the conformal infinity $(M^n,
[\hat g])$ of $(X^{n+1}, \ g)$ is locally conformally flat. Hence
the conformal infinity $(M^n, [\hat g])$ is conformally equivalent
to the standard sphere due to the simply connectedness at the
infinity. As $\bar g$ is $C^{2, \mu}$ smooth at the infinity
boundary, we know that there is a positive function $u\in C^{2,
\mu}(M^n)$ so that $u^{\frac{4}{n-3}}\bar g$ is the metric of the
standard sphere, i.e. $u$ satisfies
$$
\Delta_{\bar g}u- \frac{n-2}{4(n-1)} Ru + \frac 14 n(n-2)
u^{\frac{n+2}{n-2}}=0,
$$
By setting
$$
u(\rho,\theta)=u(\theta),
$$
we may assume $u$ is defined on $X$ at least near the infinity,
hence $u^{\frac{4}{n-2}} \cdot \bar g $ is $C^{2,\mu}$ smooth near
the infinity boundary $M$ and its restriction on $M$ is the standard
sphere metric.

Next, due to Lemma \ref{specialdefiningfunction}, there is a
geodesic defining function $x\in C^{2, \mu}$ associated with the
standard sphere metric of the conformal infinity. Therefore the
proof of Theorem 1.1 in \cite{BMQ} works here. Note that why Theorem
1.1 in \cite{BMQ} requires the regularity to be $C^{3, \alpha}$ is
due to the same assumptions in Lemma 5.1 in \cite{jL}, which has
been improved by the above Lemma \ref{specialdefiningfunction}.
\end{proof}

Theorem \ref{Einstein} is a rigidity theorem for Einstein AH
manifolds; while Theorem \ref{Riccicomparison} requires the
curvature condition (2) though it no longer needs the Einstein
equations. And in both cases the rigidity in higher dimensions
requires the spin condition. We noticed the recent work of Dutta and
Javaheri \cite{DJ} where no spin condition is assumed. The argument
in \cite{DJ} is based on the volume comparison argument in \cite{ST}
for AH manifolds with conformal compactification of $C^2$ regularity
and an additional assumption that
\begin{equation}\label{deacyofscalarcurv}
R+n(n+1) =  o(e^{-2\rho})£¬
\end{equation}
where $R$ is the scalar curvature. We will use our curvature
estimates and regularity theorems to replace the $C^2$ regularity
assumption in \cite{DJ} to prove Theorem \ref{rigidity}. Since the
proof follows the approach in \cite{DJ} with a number of
modifications, we will sketch a proof in the following for
readers$'$ conveniences. Hence from now on we will work with
asymptotically hyperbolic manifold $(X^{n+1}, \ g)$ that satisfy all
assumptions in Theorem \ref{rigidity}.

As in \cite{DJ}, let $p_0$ be any point in $X$, $t(x)$ be the
distance function to $p_0$ with respect to metric $g$, $C(p_0)$ be
the cut locus of $p_0$ in $(X,g)$, $\Sigma_t$ be the geodesic sphere
of $p_0$ with radius $t$ in $(X,g)$, $\bar g=\sinh^{-2}\rho \cdot
g$, $h=\sinh^{-2}t\cdot g$. Let  $\gamma_t$, $\eta_t$ be the
restriction metric of $\bar g$ and $h$ on $\Sigma_t$ respectively.
We continue to use $\rho$ as before to stand for the distance to the
essential set $\mathbf{E}$ and let $u = t - \rho$. It is clear that
$u$ is bounded.

Because of (\ref{estimatesp1}) in \S 2 we have Lemma 2.1 in
\cite{DJ} valid even without $C^2$ regularity of the conformal
compactness. Hence we immediately have

\begin{lemm}\label{estimate1}
There is constant $\Lambda$ which is independent of $t$ such that
$$
\|\nabla_g u \| \leq \Lambda e^{-\rho},
$$
which is equivalent to
$$
\|\nabla_{\bar g }u \| \leq \Lambda,
$$
whenever $t$ is smooth.
\end{lemm}

\begin{proof}
Let $\phi(t)=g(\nabla\rho, \nabla t)$. Then
$$
g(\nabla u, \nabla u) = 2(1- \phi).
$$
To estimate $\phi$, as in \cite{DJ}, we notice that
$$
\partial_t \phi = (1 - \phi^2)\nabla^2 \rho (n, n),
$$
where one writes $\nabla t = \phi\nabla\rho + \sqrt{1 - \phi^2} n$
and $n$ is a unit vector that is perpendicular to $\nabla\rho$.
Hence in the light of (\ref{estimatesp1}) one gets
$$
\partial_t \phi = (1 - \phi^2) ( 1 + O(e^{-2t})).
$$
By the proof of Lemma \ref{estforriccatiequa}, we then get
\begin{equation}\label{phi}
\phi = 1 + O(e^{-2t}).
\end{equation}
and finish the proof.
\end{proof}

An important consequence of the above lemma, as observed in \cite
{DJ}, is that the limit of the function $u$ is a Lipschitz function
on the infinity as $t\to\infty$, in $W^{1, p}$-norm for any $p >1$.
A geodesic in $(X, \ g)$ is said to be a $\rho$-geodesic if it is a
geodesic emanated from $\mathbf{E}$ and is perpendicular to
$\partial \mathbf{E}$; a geodesic is said to be a $t$-geodesic if it
is geodesic which is a geodesic ray from $p_0$. The following lemma
is Corollary 3.2 in \cite{DJ} which is another straightforward
consequence of the above Lemma \ref{estimate1}.

\begin{coro} There is $\rho_0>0$ such that in the region
$\rho>\rho_0$ such that the function $t(x)$ is increasing along the
$\rho$-geodesics and the function $\rho(x)$ is increasing along the
$t$-geodesics.
\end{coro}

\begin{proof}
Suppose that $x_1$ and $x_2$ are two points in a $\rho$-geodesic
with distance $s$, i.e.
$$
\rho(x_1)-\rho(x_2)=s>0.
$$
Then
\begin{equation}
\begin{split}
t(x_1)-t(x_2)&=s+u(x_1)-u(x_2)\\
& = s + s g(\nabla u, \nabla \rho)\\
&\geq s(1-\Lambda e^{-2\rho}).
\end{split}
\end{equation}
Hence there is $\rho_0$ such that $t(x)$ is increasing along
$\rho$-geodesics where $\rho > \rho_0$. Similarly we may show that
$\rho$ is increasing along $t$-geodesics where $\rho > \rho_0$ (set
$\rho_0$ bigger if necessary). Thus the proof of the lemma is
finished.
\end{proof}

Analogue to our previous rigidity theorems in this section we know
the asymptotically hyperbolic manifolds that satisfy all the
assumptions in Theorem \ref{rigidity} are conformally compact of
regularity $C^{1, \alpha}$(or $W^{2, p}$) due to our Theorem
\ref{mainresult2} and have the standard round sphere as the
conformal infinities. Particularly we know that $\partial \mathbf E$
is diffeomorphic to $\mathbf S^n$. One of the main observation in
\cite {DJ} is the following lemma, whose proof still holds with no
modification.

\begin{lemm}(Lemma 4.1 in \cite{DJ})\label{topologyofgeodesicsphere}
For $t$ large enough, $\Phi_t$: $\mathbf{S}^n \mapsto \Sigma_t$ is a
homeomorphism. Moreover it is a local diffeomorphism at $\theta \in
\mathbf{S}^n $ where $\Phi_t(\theta)\notin C(p_0)$.
\end{lemm}

In fact the set $\{\theta\in \mathbf{S}^n: \Phi_t(\theta)\notin
C(p_0)\}$ is rather negligible when we are concerned with the
integrals. Let $\mu_0$ be the standard metric on $\mathbf{S}^n$.

\begin{lemm}\label{measurezeroset}
For almost all $t$, when large enough, $\Phi^{-1}_t (\Sigma_t \cap
C(p))$ is measure zero in $(\mathbf{S}^n, \mu_0)$.
\end{lemm}

\begin{proof} At least when $\rho$ is large enough, we may consider
the map
$$
\Lambda (x) =(\Pi(x), t(x)): X\setminus E \mapsto \mathbf{S}^n
\times [0, \infty),
$$
given by the exponential map from $\partial \mathbf E$ by the nature
of an essential set and the monotonicity of the function $t$ along
each $\rho$-geodesics. Note that $\Pi$ and $t$ are Lip, so is
$\Lambda$. Therefore $\Lambda (C(p_0))$ is measure zero in
$\mathbf{S}^n \times [0, \infty)$. Due to Fubini Theorem, we see
that for almost all $t$, when large enough, $\mathbf{S}^n  \times
\{t\}\cap \Lambda (C(p_0))$ is zero measure. In the light of the
fact
$$
\Lambda|_{\Sigma_t}=\Phi^{-1}_t |_{\Sigma_t},
$$
the lemma is then proven.
\end{proof}

As argued in \cite{ST} and \cite{DJ}, due to Gromov-Bishop volume
comparison theorem, to prove Theorem \ref{rigidity} it suffice to
show
$$
\lim_{t\rightarrow \infty} Vol(\Sigma_t, \eta_t) \geq \omega_n,
$$
where $\omega_n$ is the volume of the standard sphere
$\mathbf{S}^n$. To this purpose, we study the pull back metric
$(\Phi^{-1}_t)_* \eta_t$ on $\mathbf{S}^n \setminus \Phi^{-1}_t
(C(p_0))$ as $t$ approaches to the infinity. Note that
$$
(\Phi^{-1}_t)_* \eta_t =4 e^{-2u}(\Phi^{-1}_t)_* (\bar
g|_{\Sigma_t}),
$$
and $\Sigma_t$ can be expressed as a graph $(\theta, f(\theta))$ on
$\mathbf{S}^n$.  Hence we have
$$
\frac{\partial}{\partial t}=(1+|\nabla_g
f|^2)^{-\frac12}(\frac{\partial}{\partial \rho}-g^{ij}\frac{\partial
f}{\partial \theta^i}\frac{\partial }{\partial \theta^j}),
$$
which, together with (\ref{phi}), implies
$$
|\nabla_g f|^2 =O(e^{-2\rho}).
$$
Therefore we see that
$$
(\Phi^{-1}_t)_* (\bar g|_{\Sigma_t})=\bar g_{ij}(t, \theta)d\theta^i
d\theta^j +O(e^{-2\rho}).
$$
Thus
$$
\lim_{t\rightarrow \infty}(\Phi^{-1}_t)_* ( \eta_t)
=\lim_{\rho\rightarrow \infty}4 e^{-2u}\bar
g|_{\Sigma_\rho}\triangleq \eta_0,
$$
where $\eta_0 = v^{\frac{4}{n-2}}\mu_0$ and $v$ is Lipschtz on
$\mathbf{S}^n$ satisfying
$$
n(n-1)\omega_n ^{\frac2n} \leq
\frac{\int_{\mathbf{S}^n}(\frac{4(n-1)}{(n-2)}|\nabla_{\mathbf{S}^n}
v|^2
+n(n-1)v^2)d\mu_0}{(\int_{\mathbf{S}^n}v^{\frac{2n}{n-2}}d\mu_0)^\frac{n-2}{n}},
$$
since the minimum of the Yamabe functional on $\mathbf S^n$ is
$n(n-1)\omega_n ^{\frac2n}$.

Now, on one hand, if denote $\eta_\rho =w^{\frac{4}{n-2}}\bar
g|_{\Sigma_\rho}$, $w=e^{\frac{2-n}{2}u}$, and $\bar g_\rho =\bar
g|_{\Sigma_\rho}$, we have
\begin{equation}
\begin{split}
\lim_{\rho\rightarrow \infty}\frac{\int_{\mathbf{S}^n}
R_{\eta_\rho}d\eta_\rho}{(\int_{\mathbf{S}^n}d\eta_\rho)^\frac{n-2}{n}}&=
\lim_{\rho\rightarrow
\infty}\frac{\int_{\mathbf{S}^n}(\frac{4(n-1)}{(n-2)}|\nabla_{\bar
g_\rho} w|^2 +R_{\bar g_\rho}w^2)d\bar
g_\rho}{(\int_{\mathbf{S}^n}w^{\frac{2n}{n-2}}d\bar
g_\rho)^\frac{n-2}{n}}\\
&=\frac{\int_{\mathbf{S}^n}(\frac{4(n-1)}{(n-2)}|\nabla_{\bar g_0}
w|^2 +R_{\bar g_0}w^2)d\bar
g_0}{(\int_{\mathbf{S}^n}w^{\frac{2n}{n-2}}d\bar
g_0)^\frac{n-2}{n}}\\
&=\frac{\int_{\mathbf{S}^n}(\frac{4(n-1)}{(n-2)}|\nabla_{\mathbf{S}^n}
v|^2
+n(n-1)v^2)d\mu_0}{(\int_{\mathbf{S}^n}v^{\frac{2n}{n-2}}d\mu_0)^\frac{n-2}{n}}\\
&\geq n(n-1)\omega_n ^{\frac2n}.
\end{split}
\end{equation}
Because $\bar g$ is $W^{2,p}$-regular up to the boundary of $(X,
\bar g)$ due to Theorem \ref{mainresult2} and the comment right
after the proof of Lemma \ref{estimate1}. On the other hand, by
direct computations (please see the calculations in p.556 in
\cite{ST}), we recall that,
$$
R_{\eta_\rho}\leq n(n-1)+o(1).
$$
Therefore we obtain
$$
Vol(\mathbf{S}^n, \eta_0)=\lim_{\rho\rightarrow
\infty}Vol(\mathbf{S}^n, \eta_\rho)\geq \omega_n,
$$
which implies
$$
\lim_{t\to\infty}Vol(\Sigma_t, \eta_t)\geq\omega_n.
$$
Thus the proof of Theorem \ref{rigidity} is complete.

\end{document}